\documentclass{amsart}
\usepackage{amssymb, amsthm,amsmath}
\usepackage{enumerate}
\usepackage{stmaryrd}
\usepackage{mathrsfs}
\usepackage{hyperref}
\usepackage[all]{xy}

\numberwithin{equation}{section}
\newtheorem{thm}[equation]{Theorem}
\newtheorem{lem}[equation]{Lemma}
\newtheorem{prop}[equation]{Proposition}
\newtheorem{cor}[equation]{Corollary}

\newtheorem{question}[equation]{Question}
\theoremstyle{definition}

\theoremstyle{remark}
\newtheorem{rem}[equation]{Remark}

\newtheorem{example}[equation]{Example}

\newcommand{\id}{\operatorname{id}}
\newcommand{\im}{\operatorname{im}}
\newcommand{\Hom}{\operatorname{Hom}}

\newcommand{\Ext}{\operatorname{Ext}}
\newcommand{\FP}{\operatorname{FP}}
\newcommand{\colim}{\operatorname{colim}}
\newcommand{\pd}{\operatorname{pd}}

\newcommand{\cd}{\operatorname{cd}}

\newcommand{\chac}{\operatorname{char}}
\newcommand{\Ind}{\operatorname{Ind}}
\newcommand{\cls}{{\scriptstyle\bf S}}
\newcommand{\cll}{{\scriptstyle\bf L}}
\newcommand{\clh}{{\scriptstyle\bf H}}
\newcommand{\cllh}{{\scriptstyle\bf LH}}

\begin{document}

\title{On Profinite Groups of Type $\FP_\infty$}
\address{Department of Mathematics\\
Royal Holloway, University of London\\
Egham\\
Surrey TW20 0EX\\
UK}
\email{Ged.CorobCook.2012@live.rhul.ac.uk}
\thanks{This work was done as part of the author's PhD at Royal Holloway, University of London, supervised by Brita Nucinkis.}

\begin{abstract}
Suppose $R$ is a profinite ring. We construct a large class of profinite groups $\widehat{\cll'\clh_R}\mathfrak{F}$, including all soluble profinite groups and profinite groups of finite cohomological dimension over $R$. We show that, if $G \in \widehat{\cll'\clh_R}\mathfrak{F}$ is of type $\FP_\infty$ over $R$, then there is some $n$ such that $H_R^n(G,R \llbracket G \rrbracket) \neq 0$, and deduce that torsion-free soluble pro-$p$ groups of type $\FP_\infty$ over $\mathbb{Z}_p$ have finite rank, thus answering the torsion-free case of a conjecture of Kropholler.
\end{abstract}

\maketitle

%
%

\section*{Introduction}

Recall that, given a ring $R$, a group $G$ is said to be of type $\FP_n$ over $R$ if there is a projective resolution $P_\ast$ of $R$ as an $R[G]$-module with trivial $G$-action for which $P_0, \ldots, P_n$ are finitely generated. We use a similar definition for $R$ and $G$ profinite, except that we work in the category of profinite $R \llbracket G \rrbracket$-modules and continuous homomorphisms.

In this paper we are interested in groups of type $\FP_\infty$. In the abstract case several results are already known: an important example is Kropholler's work \cite{Krop}, which shows for a large class of torsion-free groups, including all torsion-free soluble groups, that being of type $\FP_\infty$ over $R$ implies finite cohomological dimension over $R$. As far as we know, no result of this kind has previously been known for profinite groups.

In the same spirit as \cite{Krop}, we define closure operations on classes $\mathfrak{X}$ of profinite groups to construct a larger class $\widehat{\cll'\clh_R}\mathfrak{X}$, containing all profinite soluble groups, with the property that the cohomology of $\widehat{\cll'\clh_R}\mathfrak{X}$-groups relates in a manageable way to that of their $\mathfrak{X}$-subgroups. In particular, taking $\mathfrak{X}=\mathfrak{F}$, the class of finite groups with the discrete topology, we can show that if $G \in \widehat{\cll'\clh_R}\mathfrak{F}$ and $M$ is an $R \llbracket G \rrbracket$-module of type $\FP_\infty$ which is projective as an $R$-module, then there is some $n$ such that $\Ext_{R \llbracket G \rrbracket}^n(M,M \hat{\otimes}_R R \llbracket G \rrbracket) \neq 0$. Setting $M=R$ gives the result.

We get particularly nice conclusions when the group $G$ in question is torsion-free soluble pro-$p$, when some group-theoretic work along the lines of \cite[Section 2]{Krop3} shows that $G$ must have finite rank. Thus we can answer in the affirmative the torsion-free case of \cite[Open Question 6.12.1]{R-Z}, attributed to Kropholler.

Most of the paper is spent on developing the disparate elements needed to make the main result work. Some of the machinery we need is an adaptation of tools that already exist for abstract groups, but that have never been explicitly carried over before. In particular, in Section \ref{progrm} we investigate colimits of systems of profinite modules, which will arise in Sections \ref{typeIsys} and \ref{typeIIsys}. The existence of these colimits is an immediate consequence of showing that the category of profinite $R$-modules, for a topological ring $R$, is a reflective subcategory of the category of topological $R$-modules, using the usual profinite completion functor. We have been unable to find any source stating this.

Section \ref{spermmod} develops the theory of signed permutation modules: permutation modules whose $G$-action has an extra sign introduced by the process of taking tensor-induced complexes, as in \cite[7.3]{Scheiderer}. This difficulty is avoided for abstract groups by looking at cellular group actions on finite dimensional contractible complexes, which give exact sequences of (unsigned) permutation modules; in lieu of a suitable profinite analogue for these, we take a purely algebraic approach, which is slightly harder to control. Thus we need to show in particular that signed permutation modules are preserved by the constructions we will use, that is, induction and tensor-induction.

At last, in Section \ref{hierpg}, we can define the class of profinite groups that our main results will hold for. There are two operations that expand our class: $\cll'$ and $\clh$. The strategy is to apply these two alternately, infinitely many times by transfinite induction: thus groups which we can show are in the class will be ones which admit a sufficiently nice hierarchical decomposition. Thus we can show that soluble profinite groups, and elementary amenable groups of order coprime to $2$, are in the class by using their hierarchical definitions.

Sections \ref{typeIsys} and \ref{typeIIsys} deal with one of $\cll'$ and $\clh$ each: in either case, the aim is to overcome the obstacle that when the first variable is of type $\FP_\infty$, the $\Ext$-groups do not in general commute with direct limits in the second variable. Instead we show that, in the two specific situations that arise, cohomology behaves just nicely enough for our purposes. See below for how `nicely' is `nicely enough'. We put all these pieces together in Section \ref{main} to obtain the results claimed on cohomology.

In Section \ref{tdmod} we recall the work of Boggi in \cite{Boggi}, which provides some useful machinery for dealing with profinite cohomology with profinite coefficients, which allows us to deduce some group-theoretic properties in the soluble pro-$p$ case in Section \ref{solgp}.

\section{Profinite Groups, Rings and Modules}
\label{progrm}

See \cite{R-Z} or \cite{Wilson} for this background on profinite groups and profinite modules and \cite{Weibel} for background on homological algebra.

Before starting on the main body of the paper, we must establish one important fact: that categories of profinite spaces, groups and modules have all small colimits. Explicitly, these colimits are the profinite completions of the topological colimits of systems of profinite spaces/groups/modules, considered by inclusion as topological ones.

We start by fixing notation. $Set$, $Grp$, $Ab$, $Rng$ and $Mod(R)$ will be the categories of sets, groups, abelian groups, rings and left $R$-modules respectively, for $R \in Rng$. $Top$, $TGrp$, $TAb$, $TRng$ and $TMod(R)$ will be the categories of topological spaces, topological groups, topological abelian groups, topological rings and topological left $R$-modules for $R \in TRng$ respectively. Similarly $Pro$, $PGrp$, $PAb$, $PRng$ and $PMod(R)$ will be profinite spaces, profinite groups, profinite abelian groups, profinite rings and profinite left $R$-modules, for $R \in TRng$ -- we do not require $R \in PRng$. In this paper, modules will always be left modules. We also define $U$ to be the forgetful functor on each of $Top$, $TGrp$, $TAb$, $TRng$, $TMod(R)$, $Pro$, $PGrp$, $PAb$, $PRng$ and $PMod(R)$ which forgets the topology, but keeps the algebraic structure; in particular, for $R \in TRng$, $U$ sends elements of $TMod(R)$ and $PMod(R)$ to $Mod(U(R))$, rather than $Mod(R)$.

We now develop some properties of the topological categories that will help us to develop properties of the corresponding profinite categories.

\begin{lem}
\label{invtop}
Suppose $C$ is $Top$, $TGrp$, $TAb$, $TRng$ or $TMod(R)$ for $R \in TRng$. Suppose $A \in U(C)$, $B \in C$, and we have a morphism $f:A \rightarrow U(B)$ in $U(C)$. Then the collection of open sets $$T=\{f^{-1}(O): O \subseteq_{open} B\}$$ is a topology on $A$ which makes $A$ into an element of $C$.
\end{lem}
\begin{proof}
This is easy for $C=Top$.

For $C=TGrp$ or $TAb$: We write $A_T$ for $A$ endowed with topology $T$. We need to check that the multiplication map $$m_A: A_T \times A_T \rightarrow A_T$$ and the inversion map $$i_A: A_T \rightarrow A_T$$ are continuous. Given an open set $f^{-1}(O)$, $$m_A^{-1}f^{-1}(O) = (f \times f)^{-1}m_B^{-1}(O)$$ is open, and $$i_A^{-1}f^{-1}(O) = f^{-1}i_B^{-1}(O)$$ is open, as required.

For $C=TRng$: We need to check in addition to the $C=TAb$ case that multiplication is continuous. This holds in the same way as the continuity of $m$ in the $C=TGrp$ case.

For $C=TMod(R)$ for $R \in TRng$: We need to check in addition to the $C=TAb$ case that scalar multiplication $$s_A: R \times A_T \rightarrow A_T$$ is continuous. Given an open set $f^{-1}(O)$, $$s_A^{-1}f^{-1}(O) = (id_{R} \times f)^{-1}s_B^{-1}(O)$$ is open, as required.
\end{proof}

\begin{prop}
\label{limandcolim}
The categories $Top$, $TGrp$, $TAb$, $TRng$ and $TMod(R)$ for $R \in TRng$ have all small
\begin{enumerate}[(i)]
\item limits;
\item colimits.
\end{enumerate}
\end{prop}
\begin{proof}
\begin{enumerate}[(i)]
\item It is easy to check that products in any of these categories are given by endowing the product in $Set$, $Grp$, $Ab$, $Rng$ or $Mod(U(R))$ respectively with the product topology. Given a functor $F: I \rightarrow C$, where $I$ is a small category and $C$ is any of $Top$, $TGrp$, $TAb$, $TRng$ and $TMod(R)$, take the product in $C$ of the objects $F(i)$ such that $i \in I$. Now take the subobject consisting of tuples $(x_i) \in \prod_I F(i)$ such that for every morphism $f: i \rightarrow j$ in $I$ $F(f)(x_i)=x_j$, endowed with the subspace topology: one can check that this is the limit of $F$.
\item We will start from the well-known fact that $Set$, $Grp$, $Ab$, $Rng$ and $Mod(U(R))$ have all small colimits. Given a functor $F: I \rightarrow C$, where $I$ is a small category and $C$ is any of $Top$, $TGrp$, $TAb$, $TRng$ and $TMod(R)$, we can think of $F$ as an object of $C^I$, and apply the forgetful functor $U: C^I \rightarrow U(C)^I$. We know $U(F) \in U(C)^I$ has a colimit.

For $C=Top$: Let $S$ be the set of topologies on $\colim_I U^I(F)$ making all the canonical maps $$\phi_i: F(i) \rightarrow \colim_I U^I(F)$$ continuous for each $i \in I$.

For $C=TGrp$ or $TAb$: Let $S$ be the set of topologies on $\colim_I U^I(F)$ such that it is a topological group making all the canonical maps $$\phi_i: F(i) \rightarrow \colim_I U^I F(I)$$ continuous for each $i \in I$.

For $C=TRng$: Let $S$ be the set of topologies on $\colim_I U^I(F)$ such that it is a topological ring making all the canonical maps $$\phi_i: F(i) \rightarrow \colim_I U^I F(I)$$ continuous for each $i \in I$.

For $C=TMod(R)$: Let $S$ be the set of topologies on $\colim_I U^I(F)$ such that it is a topological $R$-module making all the canonical maps $$\phi_i: F(i) \rightarrow \colim_I U^I F(I)$$ continuous for each $i \in I$.

In each case we can make $S$ into a poset ordered by the fineness of the topology. Write $(\colim_I U^I(F))_{T'}$ for $\colim_I U^I(F)$ endowed with a topology $T' \in S$. By Lemma \ref{invtop} and the universal property, in each case, if $S$ has a maximal element $T$ finer than all other topologies in $S$, $(\colim_I U^I(F))_T$ is the colimit we are looking for. Now $S \neq \varnothing$ since it contains the indiscrete topology. Define $T$ to be the topology generated by the subbase $\bigcup_S T'$. We claim $T \in S$: then we will be done.

For $C=Top$: We need to check $\phi_i^{-1}(O)$ is open in $F(i)$ for each $i$ and $O$ open in $(\colim_I U^I(F))_T$. It is sufficient to check this when $O$ is in the subbase, and so open in $(\colim_I U^I(F))_{T'}$ for some $T'$, where this is clear.

For $C=TGrp$ or $TAb$: We need to check in addition to the $C=Top$ case that $(\colim_I U^I(F))_T$ is a topological group, that is, that the multiplication map $$m: (\colim_I U^I(F))_T \times (\colim_I U^I(F))_T \rightarrow (\colim_I U^I(F))_T$$ and the inversion map $$i: (\colim_I U^I(F))_T \rightarrow (\colim_I U^I(F))_T$$ are continuous. It suffices to check the inverse images of each open set $O$ in the subbase; if $O$ is open in $(\colim_I U^I(F))_{T'}$, its inverse image under $m$ is open in $$(\colim_I U^I(F))_{T'} \times (\colim_I U^I(F))_{T'},$$ hence open in $$(\colim_I U^I(F))_T \times (\colim_I U^I(F))_T;$$ its inverse image under $i$ is open in $(\colim_I U^I(F))_{T'}$, and hence it is open in $(\colim_I U^I(F))_T$.

For $C=TRng$: We need to check in addition to the $C=TAb$ case that multiplication is continuous. This holds in the same way as the continuity of $m$ in the $C=TGrp$ case.

For $C=TMod(R)$: We need to check in addition to the $C=TAb$ case that scalar multiplication $$s: R \times (\colim_I U^I(F))_T \rightarrow (\colim_I U^I(F))_T$$ is continuous. It suffices to check the inverse images of each open set $O$ in the subbase; if $O$ is open in $(\colim_I U^I(F))_{T'}$, then its inverse image is open in $R \times (\colim_I U^I(F))_{T'}$, so open in $R \times (\colim_I U^I(F))_T$.
\end{enumerate}
\end{proof}

We observe that, given a small category $I$ and a functor $F$ from $I$ to any of $Top$, $TGrp$, $TAb$, $TRng$ and $TMod(R)$, $U \colim_I F(I) = \colim_I U F(I)$ -- this can be seen from the construction of colimits in these categories.

To see the existence of similar profinite colimits, and to see that they are the profinite completions of the corresponding topological colimits, it is sufficient to show that $Pro$ is a reflective subcategory of $Top$, $PGrp$ is a reflective subcategory of $TGrp$, $PRng$ is a reflective subcategory of $TRng$, and $PMod(R)$ is a reflective subcategory of $TMod(R)$, for $R \in PRng$. Then the result follows, by \cite[Theorem 2.6.10]{Weibel}. Finally, observe that it follows from general properties of adjoint functors that profinite colimits in $PAb$ and $PMod(R)$ are right-exact, by \cite[Exercise 2.6.4]{Weibel}.

So we define a profinite completion functor on $TGrp$. The cases of $Top$, $TRng$ and $TMod(R)$ are sufficiently similar that we can safely leave them in the hands of the reader. Given $G \in TGrp$, let $N = \{U \unlhd_{\text{clopen}} G: \lvert G:U \rvert < \infty\}$. Then the profinite completion of $G$, $\hat{G}$, is the inverse limit, in $TGrp$, of the discrete finite quotients $\varprojlim_{U \in N} G/U$. It is easy to see that $\hat{}$ is a functor $TGrp \rightarrow TGrp$, by the universal property of inverse limits, and that we get a canonical morphism $\iota: G \rightarrow \hat{G}$ in $TGrp$. By the definition, $\hat{} \circ \hat{} = \hat{}$, and then we can define $PGrp$ to be the full subcategory of $TGrp$ on which $\hat{}$ is naturally isomorphic to the identity, so that in fact $\hat{}$ is a functor $TGrp \rightarrow PGrp$.

For the following lemma, for clarity, we write $t$ for the inclusion functor $PGrp \rightarrow TGrp$.

\begin{lem}
\label{pcadj}
Profinite completion satisfies the following universal property: suppose $G \in TGrp$, $H \in PGrp$. Suppose $f: G \rightarrow t(H)$ is a morphism in $TGrp$. Then $f$ factors uniquely as $$G \xrightarrow{\iota} t(\hat{G}) \xrightarrow{\hat{f}} t(H).$$ Equivalently, $\Hom_{TGrp}(G,t(H)) \cong \Hom_{PGrp}(\hat{G},H)$.
\end{lem}
\begin{proof}
This follows from the universal property of inverse limits once more, exactly as in the proof of \cite[Lemma 3.2.1]{R-Z}.
\end{proof}

This shows that $\hat{}$ is left adjoint to $t$, and hence that $PGrp$ is a reflective subcategory of $TGrp$. Similar statements apply to $Pro$, $PRng$ and $PMod(R)$. Putting all these facts together, we get the following result.

\begin{cor}
\label{pccomp}
$Pro$, $PGrp$, $PRng$ and $PMod(R)$ have all small colimits, and the colimit of a diagram in any of these categories is the profinite completion of the colimit of the same diagram in $Top$, $TGrp$, $TRng$ or $TMod(R)$, respectively.
\end{cor}
\begin{proof}
This follows immediately from Lemma \ref{pcadj} and \cite[Theorem 2.6.10]{Weibel}.
\end{proof}

\section{Signed Permutation Modules}
\label{spermmod}

From now on, $R$ will be a commutative profinite ring, $\Lambda$ a profinite $R$-algebra.

Suppose $G$ is a profinite group. Write $G \text{-} Top$ for the category of topological $G$-spaces and $G \text{-} Pro$ for the category of profinite $G$-spaces. We write elements of $G \text{-} Pro$ as $(X,\alpha)$, where $X$ is the underlying space and $\alpha:G \times X \rightarrow X$ is the action; where this is clear we may just write $X$. Now pick $X \in G \text{-} Pro$. Then the action of $G$ on $X$ induces an action of $R \llbracket G \rrbracket$ on $R \llbracket X \rrbracket$, by the universal property of group rings, making $R \llbracket X \rrbracket$ an $R \llbracket G \rrbracket$-module. We call modules with this form permutation modules, and we call the orbits and stabilisers of $G$ acting on $X$ the orbits and stabilisers of $R \llbracket X \rrbracket$. Permutation modules satisfy the following universal property: given an $R \llbracket G \rrbracket$ permutation module $R \llbracket X \rrbracket$, any continuous $G$-map from $X$ to an $R \llbracket G \rrbracket$-module $M$ factors as $X \rightarrow R \llbracket X \rrbracket \rightarrow M$ for a unique continuous $R \llbracket G \rrbracket$-homomorphism $R \llbracket X \rrbracket \rightarrow M$, where $X \rightarrow R \llbracket X \rrbracket$ is the canonical $G$-map. This can be seen by first restricting $R \llbracket X \rrbracket$ and $M$ to $PMod(R)$, making $R \llbracket X \rrbracket$ a free $R$-module, and then noting that continuous $R$-homomorphisms $R \llbracket X \rrbracket \rightarrow M$ are continuous $R \llbracket G \rrbracket$-homomorphisms if and only if they are compatible with the $G$-action. For later, we note that this universal property can be expressed by the formula in the following lemma.

\begin{lem}
\label{C_G}
Write $C_G(X,M)$ for the set of continuous $G$-maps $X \rightarrow M$. Make $C_G(X,M)$ into a $U(R)$-module by the map $r \cdot f = rf$; in other words, $(r \cdot f)(x) = r \cdot (f(x))$. Then (as $U(R)$-modules) $\Hom_{R \llbracket G \rrbracket}(R \llbracket X \rrbracket,M) \cong C_G(X,M)$.
\end{lem}
\begin{proof}
That $\Hom_{R \llbracket G \rrbracket}(R \llbracket X \rrbracket,M)$ and $C_G(X,M)$ are isomorphic as sets is simply a restatement of the universal property. That they have the same $U(R)$-module structure is clear from the definition of multiplication by $r$.
\end{proof}

Signed permutation modules are $R \llbracket G \rrbracket$-modules which as $R$-modules are free with basis $X$, and whose $G$-action comes from a continuous action of $G$ on $X \cup -X \subset R \llbracket X \rrbracket$ such that $g \cdot -x = -(g \cdot x)$ for all $g \in G, x \in X \cup -X$; the terminology appears in \cite[Definition 5.1]{Symonds}, though in fact the definitions are slightly different: both this definition and that of \cite{Symonds} are attempts to deal with the `twist' by a sign that appears in the tensor-induced complexes of \cite[7.3]{Scheiderer}. The reason for the change is that our definition seems to be needed for Lemma \ref{indsperm}.

In the same way as for permutation modules, one can see that signed permutation modules satisfy the property that any continuous $G$-map $f$ from $X \cup -X$ to an $R \llbracket G \rrbracket$-module $M$ such that $f(-x)=-f(x)$ for all $x$ extends uniquely to a continuous $R \llbracket G \rrbracket$-homomorphism $R \llbracket X \rrbracket \rightarrow M$.

For this paragraph, assume $\chac R \neq 2$. Now suppose $P$ is a signed permutation module of the form $R \llbracket X \rrbracket$. Write $\overline{X}$ for the quotient $G$-space of $X \sqcup -X$ given by $x \sim -x$, and $\sim$ for the map $X \sqcup -X \rightarrow \overline{X}$. Then we make the convention that when we talk about the $G$-stabilisers of $P$, we will always mean the $G$-stabilisers of $\overline{X}$, and the $G$-orbits of $P$ will always mean the preimages in $X$ of the $G$-orbits of $\overline{X}$.

If on the other hand $\chac R = 2$, signed permutation modules are just permutation modules. So here the $G$-stabilisers of $R \llbracket X \rrbracket$ will be the $G$-stabilisers of $X$ and the $G$-orbits will be the $G$-orbits of $X$. We also define, for $\chac R = 2$, $\overline{X} = X$ and $\sim = \id_X$. We use the notation that $G$ acts on $X \cup -X$ to cover both cases.

We now need to establish some basic properties of signed permutation modules. The following lemma is an adaptation of \cite[Lemma 5.6.4(a)]{R-Z}.

\begin{lem}
\label{algstruct}
Suppose $R \llbracket X \rrbracket$ is a signed permutation module. Then $X \cup -X = \varprojlim (X_i \cup -X_i)$, where the $X_i \cup -X_i$ are finite quotients of $X \cup -X$ as $G$-spaces such that the map $X \cup -X \rightarrow X_i \cup -X_i$ sends $X$ to $X_i$ and $-X$ to $-X_i$. Thus $$R \llbracket X \rrbracket = \varprojlim_{PMod(R \llbracket G \rrbracket)} R[X_i] = \varprojlim_{PMod(R \llbracket G \rrbracket)} R_j[X_i],$$ where the $R_j$ are the finite quotients of $R$. We say that such quotients $R_j[X_i]$ of $R \llbracket X \rrbracket$ preserve the algebraic structure.
\end{lem}
\begin{proof}
If $\chac R = 2$, we are done. Assume $\chac R \neq 2$.

Consider the set $S$ of clopen equivalence relations $\mathcal{R}$ on $X \cup -X$ such that, considered as a subset of $(X \cup -X) \times (X \cup -X)$, $$\mathcal{R} \subseteq (X \times X) \cup (-X \times -X)$$ and $$(x,y) \in \mathcal{R} \Leftrightarrow (-x,-y) \in \mathcal{R}.$$ In other words, an equivalence relation $\mathcal{R} \in S$ is one which does not identify anything in $X$ with anything in $-X$ and identifies two elements in $-X$ whenever it identifies the corresponding two elements of $X$; then $\mathcal{R} \in S$ if and only if $(X \cup -X)/\mathcal{R}$ has the form $X_i \cup -X_i$ for some finite quotient $X_i$ of $X$ (as profinite spaces, not profinite $G$-spaces). Since $X = \varprojlim_{Pro} X_i$, $$X \cup -X = \varprojlim_{Pro} X_i \cup -X_i = \varprojlim_{Pro,S} (X \cup -X)/\mathcal{R}.$$ We want to show that for every $\mathcal{R} \in S$ there is some $\mathcal{R}' \subseteq \mathcal{R}$ which is $G$-invariant: then it follows that $$X \cup -X = \varprojlim_{G \text{-} Pro,\{\mathcal{R} \in S: \mathcal{R} \text{ is } G \text{-invariant}\}} (X \cup -X)/\mathcal{R}$$ by \cite[Lemma 1.1.9]{R-Z}, because $\{\mathcal{R} \in S: \mathcal{R} \text{ is } G \text{-invariant}\}$ is cofinal in $S$, and all these quotients are quotients as $G$-spaces.

So suppose $\mathcal{R} \in S$ and define $\mathcal{R}' = \bigcap_{g \in G} g\mathcal{R}$, where $$g\mathcal{R} = \{(gx,gy) \in (X \cup -X) \times (X \cup -X): (x,y) \in \mathcal{R}\}.$$ Now we see in exactly the same way as the proof of \cite[Lemma 5.6.4(a)]{R-Z} that $\mathcal{R}'$ is clopen; clearly $\mathcal{R}'$ is $G$-invariant, and $\mathcal{R}' \in S$ because $$\mathcal{R}' \subseteq \mathcal{R} \subseteq (X \times X) \cup (-X \times -X)$$ and
\begin{align*}
(x,y) \in \mathcal{R}' &\Leftrightarrow (x,y) \in g\mathcal{R}, \forall g \\
&\Leftrightarrow (g^{-1}x,g^{-1}y) \in \mathcal{R}, \forall g \\
&\Leftrightarrow (-g^{-1}x,-g^{-1}y) \in \mathcal{R}, \forall g \\
&\Leftrightarrow (-x,-y) \in g\mathcal{R}, \forall g \\
&\Leftrightarrow (-x,-y) \in \mathcal{R}'.
\end{align*}

It follows that $R \llbracket X \rrbracket = \varprojlim_{PMod(R \llbracket G \rrbracket)} R_j[X_i]$ because every continuous $G$-map $f$ from $X \cup -X$ to a finite $R \llbracket G \rrbracket$-module $M$ such that $f(-x)=-f(x)$ for all $x$ factors through some quotient $G$-space $X_i \cup -X_i$, and clearly the induced map $f': X_i \cup -X_i \rightarrow M$ satisfies $f'(-x)=-f'(x)$, so every morphism $R \llbracket X \rrbracket \rightarrow M$ factors through some $R[X_i]$ by the universal property of signed permutation modules, and hence through some $R_j[X_i]$.
\end{proof}

\begin{lem}
\label{freeperm}
Suppose $R \llbracket X \rrbracket$ is a signed $R \llbracket G \rrbracket$ permutation module, and that $G$ acts freely on $\overline{X}$. Then $R \llbracket X \rrbracket$ is free.
\end{lem}
\begin{proof}
If $\chac R = 2$, we are done. Assume $\chac R \neq 2$.

As profinite $G$-spaces, $\overline{X} = G \times \overline{Y}$ for some $\overline{Y}$ on which $G$ acts trivially by \cite[Corollary 5.6.6]{R-Z}; take the preimage $Y$ of $\overline{Y}$ in $X$. Then we want to show $R \llbracket X \rrbracket$ is a free $R \llbracket G \rrbracket$-module with basis $Y$. Now $G$ acts freely on $G \times Y$, so by the universal property of free $R$-modules it is enough to show that $X \cup -X = G \times Y \cup -(G \times Y)$ as topological spaces. The inclusion $Y \rightarrow X \cup -X$ gives a continuous map $$G \times Y \rightarrow G \times (X \cup -X) \xrightarrow{\cdot} X \cup -X$$ and similarly for $-(G \times Y)$, after multiplying by $-1$. Hence we get a continuous map $G \times Y \cup -(G \times Y) \rightarrow X \cup -X$ which is bijective by the choice of $Y$, so the two are homeomorphic because they are compact and Hausdorff.
\end{proof}

Permutation modules behave nicely with respect to induction of modules; we want to show the same is true of signed permutation modules.

We first recall the definition of induction: on $H$-spaces, for $H \leq G$, we define $\Ind^G_H$ by the universal property that, given $X \in H \text{-} Pro$, $X' \in G \text{-} Pro$ and a continuous map $f:X \rightarrow X'$ as $H$-spaces, $f$ factors uniquely through a map $f':\Ind^G_HX \rightarrow X'$ of $G$-spaces. Clearly $\Ind^G_HX$ is unique up to isomorphism. In addition this property makes $\Ind^G_H$ a functor in the obvious way. Analogously, given $A \in PMod(R \llbracket H \rrbracket), B \in PMod(R \llbracket G \rrbracket)$, $\Ind^G_H$ is defined by the universal property that a morphism $f:A \rightarrow B$ in $PMod(R \llbracket H \rrbracket)$ factors uniquely through $f':\Ind^G_HA \rightarrow B$ in $PMod(R \llbracket G \rrbracket)$.

Recall also that, given $H \leq G$, it is possible to choose a closed left transversal $T$ of $H$ by \cite[Proposition 1.3.2]{Wilson}. In other words, $T$ is a closed subset of $G$ containing exactly one element of each left coset of $H$ in $G$. By \cite[Proposition 1.3.4]{Wilson} we then have a homeomorphism $G \cong T \times H$ as spaces.

\begin{lem}
\label{indsperm}
Let $H \leq G$, and suppose $R \llbracket X \rrbracket \in PMod(R \llbracket H \rrbracket)$ is a signed permutation module. Then $\Ind^G_H R \llbracket X \rrbracket$ is a signed permutation module in $PMod(R \llbracket G \rrbracket)$.
\end{lem}
\begin{proof}
Assume $\chac R \neq 2$; otherwise we are done.

We know that $X \cup -X \in H \text{-} Pro$, and any composite map $f: X \cup -X \rightarrow R \llbracket X \rrbracket \rightarrow M$ for $M \in PMod(R \llbracket G \rrbracket)$ satisfies $f(-x)=-f(x)$ for all $x \in X$.

Now $\Ind^G_H (X \cup -X)$ can be constructed in the following way: choose a closed left transversal $T$ of $H$ in $G$ and take the space $T \times (X \cup -X)$ with the product topology. Every element of $G$ can be written uniquely in the form $th$ with $t \in T, h \in H$. So given $g \in G, t \in T$, write $gt$ in the form $t'h$, $t' \in T, h \in H$ and define $g \cdot (t,x) = (t',h \cdot x)$. This gives an abstract group action on $T \times (X \cup -X)$ because, if $g_2t=t'h_2$ and $g_1t'=t''h_1$, for $g_1, g_2 \in G, t,t',t'' \in T, h_1, h_2 \in H$, $g_1g_2t=t''h_1h_2$ and hence $$g_1 \cdot (g_2 \cdot (t,x)) = g_1 \cdot (t',h_2 \cdot x) = (t'',h_1 \cdot(h_2 \cdot x)) = (t'',(h_1h_2)\cdot x) = (g_1g_2) \cdot (t,x);$$ to see the action is continuous, note that we can write the action as the following composite of continuous maps:
\begin{align*}
G \times T \times (X \cup -X) &\xrightarrow{m \times \id} G \times (X \cup -X) \\
&\xrightarrow{\theta \times \id} T \times H \times (X \cup -X) \\
&\xrightarrow{\id \times \alpha} T \times (X \cup -X).
\end{align*}
Here $m$ is multiplication in $G$, $\theta$ is the homeomorphism $G \rightarrow T \times H$, and $\alpha$ is the $H$-action on $X \cup -X$. We claim that the space $T \times (X \cup -X)$, with this $G$-action, satisfies the universal property to be $\Ind^G_H (X \cup -X)$, where the canonical map $X \cup -X \rightarrow \Ind^G_H (X \cup -X)$ is given by $x \mapsto (1,x)$. Indeed, given $M \in PMod(R \llbracket G \rrbracket)$ and a continuous map $$f: X \cup -X \rightarrow M$$ of $H$-spaces such that $f(-x)=-f(x)$ for all $x \in X$, define $$f': T \times (X \cup -X) \rightarrow M, f':(t,x) \mapsto t \cdot f(x):$$ this is a $G$-map because, for $gt=t'h$, $g \in G, t, t' \in T, h \in H$, $$f'(g \cdot (t,x)) = f'(t',h \cdot x) = t' \cdot f(h \cdot x) = (t'h) \cdot f(x) = g \cdot (t \cdot f(x)).$$ The uniqueness of this choice of map is clear. Moreover, we have $$f'(t,-x) = t \cdot f(-x) = t \cdot (-f(x)) = -(t \cdot f(x)) = -f'(t,x),$$ and hence by the universal property of signed permutation modules $f'$ extends uniquely to a morphism $R \llbracket T \times X \rrbracket \rightarrow M$, where $R \llbracket T \times X \rrbracket$ is the signed permutation module with the $G$-action on $T \times X \cup -(T \times X)$ given by the $G$-action on $T \times (X \cup -X)$. By the universal property of induced modules this $R \llbracket T \times X \rrbracket$ is $\Ind^G_H R \llbracket X \rrbracket$.
\end{proof}

If $R \llbracket X \rrbracket$ is a signed $R \llbracket G \rrbracket$ permutation module with $\overline{X} \cong G/H$ as $G$-spaces, we may write $R \llbracket X \rrbracket = R \llbracket G/H;\sigma \rrbracket$, where $\sigma$ is the $G$-action on $X \cup -X$, with the understanding that $G$ acts on $G/H \cup -G/H$, for each $g \in G, tH \in G/H$, by either $\sigma(g,tH) = gtH$ or $\sigma(g,tH) = -gtH$ (and similarly for $-tH \in -G/H$). When there is no ambiguity we may simply write $R \llbracket G/H \rrbracket$ for this. In particular, each element of $H$ acts on the cosets $1H \cup -1H$ by multiplication by $\pm 1$, giving a continuous homomorphism $\varepsilon: H \rightarrow \{\pm 1\}$, which we will refer to as the twist homomorphism of $R \llbracket G/H; \sigma \rrbracket$.

\begin{lem}
\label{twist}
Write $R'$ for a copy of $R$ on which $H$ acts by $h \cdot r = \varepsilon(h)r$. Then we have $\Ind^G_H R' = R \llbracket G/H;\sigma \rrbracket$.
\end{lem}
\begin{proof}
Assume $\chac R \neq 2$; otherwise we are done.

By Lemma \ref{indsperm} we have that $\Ind^G_H R' = R \llbracket \Ind^G_H \{\pm 1\} \rrbracket$. We will show that, as $G$-spaces, $$\Ind^G_H(\{\pm 1\},\varepsilon) \cong (G/H \cup -G/H,\sigma).$$ Now by the choice of $\varepsilon$ we have a continuous map $$f: \{\pm 1\} \rightarrow G/H \cup -G/H, \pm 1 \mapsto \pm 1H$$ of $H$-spaces satisfying $f(-x)=-f(x)$, and the proof of Lemma \ref{indsperm} gives us a continuous map $f':\Ind^G_H\{\pm 1\} \rightarrow G/H \cup -G/H$ of $G$-spaces extending this, such that $f'(-x)=-f'(x)$. Explicitly, choosing a closed left transversal $T$ of $H$ as before, $\Ind^G_H\{\pm 1\} = T \cup -T$, and $f'(t)=\sigma(t,1H)$ for $t \in T$, $f'(t)=\sigma(t,-1H)$ for $t \in -T$. Now $f'$ is bijective because every element of $G/H \cup -G/H$ can be written uniquely in the form $\sigma(t,1H)$ or $\sigma(t,-1H)$ for some $t \in T$. Therefore $f'$ is a homeomorphism, and we are done.
\end{proof}

Finally, we justify our introduction of signed permutation modules, instead of permutation modules. As stated at the beginning of the section, they are an attempt to deal with the tensor-induced complexes of \cite[7.3]{Scheiderer}. We sketch the construction of these complexes.

To fix notation, we start by defining wreath products. Given $G \in PGrp$, let $G^n$ be the direct product in $PGrp$ of $n$ copies of $G$. Let $S_n$ be the symmetric group on $n$ letters, acting on the right. Then the wreath product of $G$ by $S_n$, written $G \wr S_n$, is the semidirect product of $G^n$ and $S_n$, where $S_n$ acts by permuting the copies of $G$. More explicitly, we can write $G \wr S_n$ as $G^n \times S_n$ as a space, with group operation $$(h_1, \ldots, h_n, \pi) \cdot (h'_1, \ldots, h'_n, \pi') = (h_1h'_{1\pi}, \ldots, h_nh'_{n\pi}, \pi\pi').$$ Since the action of $S_n$ on $G^n$ is continuous, this makes $G \wr S_n$ a topological group, which is then profinite because it is compact, Hausdorff and totally disconnected.

Suppose $G \in PGrp$. Let $P_\ast$ be a non-negative complex of profinite $R \llbracket G \rrbracket$-modules. Then one can take the $n$-fold tensor power of $P_\ast$, $P_\ast^{\hat{\otimes} n}$, over $R$ by defining $$P_k^{\hat{\otimes} n} = \bigoplus_{i_1 + \cdots + i_n = k} P_{i_1} \hat{\otimes}_R \cdots \hat{\otimes}_R P_{i_n},$$ with the differential maps coming from repeated use of the sign trick in \cite[1.2.5]{Weibel}: this gives a non-negative complex of profinite $R$-modules. Moreover, by \cite[7.3]{Scheiderer}, it can be made into a complex of $R \llbracket G \wr S_n \rrbracket$-modules by the $G \wr S_n$-action $$(h_1 , \ldots, h_n, \pi) \cdot (q_1 \hat{\otimes} \cdots \hat{\otimes} q_n) = (-1)^\nu \cdot h_1q_{1\pi} \hat{\otimes} \cdots \hat{\otimes} h_nq_{n\pi}$$ where the $q_i \in P_\ast$ are homogeneous elements and $\nu$ is the integer $$\nu = \sum_{i<j,i\pi>j\pi} \deg(q_{i\pi})\deg(q_{j\pi}).$$

We can now generalise \cite[7.4]{Scheiderer} slightly -- the proof is largely the same.

\begin{prop}
\label{tenind}
Suppose $$\cdots \rightarrow R \llbracket X_1 \rrbracket \rightarrow R \llbracket X_0 \rrbracket \rightarrow 0$$ is an exact sequence in $PMod(R \llbracket G \rrbracket)$ of signed permutation modules, and write $P_\ast$ for this chain complex. Then $P_\ast^{\hat{\otimes} n}$ is an exact sequence of signed permutation modules in $PMod(R \llbracket G \wr S_n \rrbracket)$.
\end{prop}
\begin{proof}
Assume $\chac R \neq 2$; the proof for $\chac R = 2$ is similar.

First note that each module in $P_\ast$ is free as an $R$-module, so, for each $i$, $R \llbracket X_i \rrbracket \hat{\otimes}_R -$ is an exact functor on $R$-modules, and hence $P_\ast^{\hat{\otimes} n}$ is exact by $n-1$ applications of \cite[Lemma 2.7.3]{Weibel}. Now as $R$-modules one has $$P_k^{\hat{\otimes} n} = \bigoplus_{i_1 + \cdots + i_n = k} R \llbracket X_{i_1} \rrbracket \hat{\otimes}_R \cdots \hat{\otimes}_R R \llbracket X_{i_n} \rrbracket = R \llbracket \bigsqcup_{i_1 + \cdots + i_n = k} X_{i_1} \times \cdots \times X_{i_n} \rrbracket$$ by \cite[Exercise 5.5.5(a)]{R-Z}, so we simply need to show that $$\bigsqcup_{i_1 + \cdots + i_n = k} X_{i_1} \times \cdots \times X_{i_n} \cup -(\bigsqcup_{i_1 + \cdots + i_n = k} X_{i_1} \times \cdots \times X_{i_n})$$ is a $G \wr S_n$-subspace of $R \llbracket \bigsqcup_{i_1 + \cdots + i_n = k} X_{i_1} \times \cdots \times X_{i_n} \rrbracket$. For an element $x_1 \hat{\otimes} \cdots \hat{\otimes} x_n$ of this subspace, so that $x_j \in X_{i_j} \cup -X_{i_j}$ for each $j$, we have $$(h_1, \ldots, h_n, \pi)\cdot (x_1 \hat{\otimes} \cdots \hat{\otimes} x_n) = (-1)^\nu \cdot h_1x_{1\pi} \hat{\otimes} \cdots \hat{\otimes} h_nx_{n\pi},$$ and then $i_{1\pi} + \cdots + i_{n\pi} = i_1 + \cdots i_n = k$. Moreover, for each $j$ we have $$x_{j\pi} \in X_{i_{j\pi}} \cup -X_{i_{j\pi}} \Rightarrow h_jx_{j\pi} \in X_{i_{j\pi}} \cup -X_{i_{j\pi}},$$ as required.
\end{proof}

\section{A Hierarchy of Profinite Groups}
\label{hierpg}

We define classes of groups and closure operations on them as in \cite{Krop}, except that all our groups are required to be profinite and all maps continuous. Thus, for example, all our subgroups will be assumed to be closed unless stated otherwise. As there, for a class of profinite groups $\mathfrak{X}$, we let $\cls\mathfrak{X}$ be the class of closed subgroups of groups in $\mathfrak{X}$, and $\cll\mathfrak{X}$ be those profinite groups $G$ such that every finite subset of $G$ is contained in some subgroup $H \leq G$ with $H \in \mathfrak{X}$. We also define a more general version $\cll'$ of $\cll$: $\cll'\mathfrak{X}$ is the class of profinite groups $G$ which have a direct system of subgroups $\{G_i\}$, ordered by inclusion, whose union is dense in $G$, such that $G_i \in \mathfrak{X}$ for every $i$. Given two classes $\mathfrak{X}$ and $\mathfrak{Y}$, we write $\mathfrak{XY}$ for extensions of a group in $\mathfrak{X}$ by a group in $\mathfrak{Y}$.

Lastly, we define $\clh_R\mathfrak{X}$ to be the profinite groups $G$ for which there is an exact sequence $0 \rightarrow P_n \rightarrow \cdots \rightarrow P_0 \rightarrow R \rightarrow 0$ of $R \llbracket G \rrbracket$-modules, where, for each $i$, $P_i$ is a signed permutation module, all of whose stabilisers are in $\mathfrak{X}$. We will refer to this as a finite length signed permutation resolution of $G$.

Note that $\clh_R$ is not a closure operation. Instead, we use it to define inductively the class of groups $(\clh_R)_\alpha\mathfrak{X}$ for each ordinal $\alpha$: $(\clh_R)_0\mathfrak{X} = \mathfrak{X}$, $(\clh_R)_\alpha\mathfrak{X} = \clh_R((\clh_R)_{\alpha-1}\mathfrak{X})$ for $\alpha$ a successor, and $(\clh_R)_\alpha\mathfrak{X} = \bigcup_{\beta < \alpha} (\clh_R)_\beta\mathfrak{X}$ for $\alpha$ a limit. Finally, we write $\widehat{\clh_R}\mathfrak{X} = \bigcup_\alpha (\clh_R)_\alpha\mathfrak{X}$. It is easy to check that $\widehat{\clh_R}$ is a closure operation.

Similarly we can define $(\cllh_R)_0\mathfrak{X} = \mathfrak{X}$ and $(\cll'\clh_R)_0\mathfrak{X} = \mathfrak{X}$, then $(\cllh_R)_\alpha\mathfrak{X} = \cllh_R((\cllh_R)_{\alpha-1}\mathfrak{X})$ and $(\cll'\clh_R)_\alpha\mathfrak{X} = \cll'\clh_R((\cll'\clh_R)_{\alpha-1}\mathfrak{X})$ for $\alpha$ a successor, and finally $(\cllh_R)_\alpha\mathfrak{X} = \bigcup_{\beta < \alpha} (\cllh_R)_\beta\mathfrak{X}$ and $(\cll'\clh_R)_\alpha\mathfrak{X} = \bigcup_{\beta < \alpha} (\cll'\clh_R)_\beta\mathfrak{X}$ for $\alpha$ a limit. Then let $\widehat{\cllh_R}\mathfrak{X} = \bigcup_\alpha (\cllh_R)_\alpha\mathfrak{X}$ and $\widehat{\cll'\clh_R}\mathfrak{X} = \bigcup_\alpha (\cll'\clh_R)_\alpha\mathfrak{X}$: this gives two more closure operations with $\widehat{\clh_R}\mathfrak{X} \leq \widehat{\cllh_R}\mathfrak{X} \leq \cll'\widehat{\cllh_R}\mathfrak{X} \leq \widehat{\cll'\clh_R}\mathfrak{X}$ for all $\mathfrak{X}$. The final inequality holds because $\cll'\widehat{\cllh_R}\mathfrak{X} \leq \cll'\widehat{\cll'\clh_R}\mathfrak{X}$ and $$\cll'(\cll'\clh_R)_\alpha\mathfrak{X} \leq \cll'\clh(\cll'\clh_R)_\alpha\mathfrak{X} = (\cll'\clh_R)_{\alpha+1}\mathfrak{X}, \forall \alpha \Rightarrow \cll'\widehat{\cll'\clh_R}\mathfrak{X} = \widehat{\cll'\clh_R}\mathfrak{X}.$$

\begin{rem}
In the abstract case, \cite[2.2]{Krop} shows that any countable $\widehat{\cllh_R}\mathfrak{X}$-group is actually in $\widehat{\clh_R}\mathfrak{X}$, greatly diminishing the importance of $\cll$, inasfar as the hierarchy is used to study finitely generated groups. The same argument does not work for profinite groups.
\end{rem}

From now on, $\mathfrak{F}$ will mean the class of finite groups, and $\mathfrak{I}$ the class of the trivial group.

\begin{prop}
\label{subext+}
Let $\mathfrak{X}$ be a class of profinite groups.
\begin{enumerate}[(i)]
\item $\cls\widehat{\clh_R}\mathfrak{X} \leq \widehat{\clh_R}\cls\mathfrak{X}$.
\item $(\widehat{\clh_R}\cls\mathfrak{X})\mathfrak{F} \leq \widehat{\clh_R}\cls(\mathfrak{XF})$.
\item $(\widehat{\clh_R}\mathfrak{F})(\widehat{\clh_R}\mathfrak{F}) = \widehat{\clh_R}\mathfrak{F}$.
\end{enumerate}
\end{prop}
\begin{proof}
\begin{enumerate}[(i)]
\item Use induction on $\alpha$. We will show that $\cls(\clh_R)_\alpha\mathfrak{X} \leq (\clh_R)_\alpha\cls\mathfrak{X}$ for each $\alpha$. The case when $\alpha$ is $0$ or a limit ordinal is trivial. Suppose $G \in \cls(\clh_R)_{\alpha+1}\mathfrak{X}$ and pick $H \in (\clh_R)_{\alpha+1}\mathfrak{X}$ with $G \leq H$. Take a finite length signed permutation resolution of $H$ with stabilisers in $(\clh_R)_\alpha\mathfrak{X}$. Restricting this resolution to $G$ gives a finite length signed permutation resolution whose stabilisers are subgroups of the stabilisers in the original resolution of $H$, so the stabilisers are in $\cls(\clh_R)_\alpha\mathfrak{X} \leq (\clh_R)_\alpha\cls\mathfrak{X}$, where the inequality holds by our inductive hypothesis, and hence $G \in (\clh_R)_{\alpha+1}\cls\mathfrak{X}$.
\item Use induction on $\alpha$. We will show that $((\clh_R)_\alpha\cls\mathfrak{X})\mathfrak{F} \leq (\clh_R)_\alpha\cls(\mathfrak{XF})$ for each $\alpha$. The case when $\alpha$ is $0$ or a limit ordinal is trivial.  So suppose $G \in ((\clh_R)_{\alpha+1}\cls\mathfrak{X})\mathfrak{F}$, and suppose $H \unlhd_{\text{open}} G$, $H \in (\clh_R)_{\alpha+1}\cls\mathfrak{X}$. Take a finite length signed permutation resolution of $H$ with stabilisers in $(\clh_R)_\alpha\cls\mathfrak{X}$. Then we get a finite length signed permutation resolution of $H \wr S_{\lvert G/H \rvert}$ by Proposition \ref{tenind}. Moreover, $G$ embeds in $H \wr S_{\lvert G/H \rvert}$ by \cite[7.1]{Scheiderer}, so by restriction this is also a finite length signed permutation resolution of $G$. Finally, it is clear from the construction that the stabilisers under the $G$-action are all finite extensions of subgroups of stabilisers in the original resolution of $H$, which are in $(\clh_R)_\alpha\cls\mathfrak{X}$ by (i) and our inductive hypothesis. Therefore this tensor-induced complex shows that $G \in (\clh_R)_{\alpha+1}\cls\mathfrak{XF}$.
\item Use induction on $\alpha$. We will show that $(\widehat{\clh_R}\mathfrak{F})((\clh_R)_\alpha\mathfrak{F}) \leq \widehat{\clh_R}\mathfrak{F}$ for each $\alpha$. The other inequality is clear. The case when $\alpha$ is a limit ordinal is trivial; the case $\alpha=0$ holds by (ii). Suppose $G \in (\widehat{\clh_R}\mathfrak{F})((\clh_R)_{\alpha+1}\mathfrak{F})$ and pick $H \unlhd G$ such that $H \in \widehat{\clh_R}\mathfrak{F}$ and $G/H \in (\clh_R)_{\alpha+1}\mathfrak{F}$. Take a finite length signed permutation resolution of $G/H$ with stabilisers in $(\clh_R)_\alpha\mathfrak{F}$. Restricting this resolution to $G$ gives a finite length signed permutation resolution whose stabilisers are extensions of $H$ by stabilisers in the original resolution of $G/H$, so the stabilisers are in $(\widehat{\clh_R}\mathfrak{F})((\clh_R)_\alpha\mathfrak{F}) \leq \widehat{\clh_R}\mathfrak{F}$, where the inequality holds by our inductive hypothesis, and hence $G \in \widehat{\clh_R}\mathfrak{F}$.
\end{enumerate}
\end{proof}

\begin{prop}
\label{subext}
Let $\mathfrak{X}$ be a class of profinite groups.
\begin{enumerate}[(i)]
\item $\cls\widehat{\cllh_R}\mathfrak{X} \leq \widehat{\cllh_R}\cls\mathfrak{X}$.
\item $(\widehat{\cllh_R}\cls\mathfrak{X})\mathfrak{F} \leq \widehat{\cllh_R}\cls(\mathfrak{XF})$.
\item $(\widehat{\cllh_R}\mathfrak{F})(\widehat{\cllh_R}\mathfrak{F}) = \widehat{\cllh_R}\mathfrak{F}$.
\end{enumerate}
\end{prop}
\begin{proof}
\begin{enumerate}[(i)]
\item Use induction on $\alpha$. We will show $\cls\clh_R(\cllh_R)_\alpha\mathfrak{X} \leq \clh_R(\cllh_R)_\alpha\cls\mathfrak{X}$ and hence that $\cls(\cllh_R)_{\alpha+1}\mathfrak{X} \leq (\cllh_R)_{\alpha+1}\cls\mathfrak{X}$ for each $\alpha$. The case when $\alpha$ is $0$ or a limit ordinal is trivial. Suppose first that $G_1 \in \cls\clh_R(\cllh_R)_\alpha\mathfrak{X}$ and pick $H_1 \in \clh_R(\cllh_R)_\alpha\mathfrak{X}$ with $G_1 \leq H_1$. Take a finite length signed permutation resolution of $H_1$ with stabilisers in $(\cllh_R)_\alpha\mathfrak{X}$. Restricting this resolution to $G_1$ gives a finite length signed permutation resolution whose stabilisers are subgroups of the stabilisers in the original resolution of $H_1$, so the stabilisers are in $\cls(\cllh_R)_\alpha\mathfrak{X} \leq (\cllh_R)_\alpha\cls\mathfrak{X}$, where the inequality holds by our inductive hypothesis, and hence $G_1 \in \clh_R(\cllh_R)_\alpha\cls\mathfrak{X}$. Suppose next that $G_2 \in \cls(\cllh_R)_{\alpha+1}\mathfrak{X}$ and pick $H_2 \in (\cllh_R)_{\alpha+1}\mathfrak{X}$ with $G_2 \leq H_2$. Every finitely generated subgroup of $H_2$ is contained in some $K \leq H_2$ with $K \in \clh_R(\cllh_R)_\alpha\mathfrak{X}$, so every finitely generated subgroup of $H_2$ is in $\clh_R(\cllh_R)_\alpha\cls\mathfrak{X}$ by our inductive hypothesis. In particular this is true for the finitely generated subgroups of $G_2$, and hence $G_2 \in (\cllh_R)_{\alpha+1}\cls\mathfrak{X}$.
\item Use induction on $\alpha$. The case when $\alpha$ is $0$ or a limit ordinal is trivial. We will show that $((\cllh_R)_\alpha\cls\mathfrak{X})\mathfrak{F} \leq (\cllh_R)_\alpha\cls(\mathfrak{XF})$ for each $\alpha$. So suppose $G \in ((\cllh_R)_{\alpha+1}\cls\mathfrak{X})\mathfrak{F}$, and suppose $H \unlhd_{\text{open}} G$, $H \in (\cllh_R)_{\alpha+1}\cls\mathfrak{X}$. It suffices to prove that every finitely generated subgroup of $G$ belongs to $\clh_R(\cllh_R)_\alpha\cls\mathfrak{X}$, and so we may assume that $G$ is finitely generated. This implies $H$ is finitely generated, by \cite[Proposition 2.5.5]{R-Z}, so $H \in \clh_R(\cllh_R)_\alpha\cls\mathfrak{X}$. Take a finite length signed permutation resolution of $H$ with stabilisers in $(\cllh_R)_\alpha\cls\mathfrak{X}$. Then we get a finite length signed permutation resolution of $H \wr S_{\lvert G/H \rvert}$ by Proposition \ref{tenind}. Moreover, $G$ embeds in $H \wr S_{\lvert G/H \rvert}$ by \cite[7.1]{Scheiderer}, so by restriction this is also a finite length signed permutation resolution of $G$. Finally, it is clear from the construction that the stabilisers under the $G$-action are all finite extensions of subgroups of stabilisers in the signed permutation resolution of $H$, which are in $(\cllh_R)_\alpha\cls\mathfrak{X}$ by (i) and our inductive hypothesis. Therefore this tensor-induced complex shows that $G \in \clh_R(\cllh_R)_\alpha\cls\mathfrak{X}$.
\item Use induction on $\alpha$. We will show that $(\widehat{\cllh_R}\mathfrak{F})((\cllh_R)_\alpha\mathfrak{F}) \leq \widehat{\cllh_R}\mathfrak{F}$ for each $\alpha$. The other inequality is clear. The case when $\alpha$ is a limit ordinal is trivial; the case $\alpha=0$ holds by (ii). Suppose $G \in (\widehat{\cllh_R}\mathfrak{F})((\cllh_R)_{\alpha+1}\mathfrak{F})$ and pick $H \unlhd G$ such that $H \in \widehat{\cllh_R}\mathfrak{F}$ and $G/H \in (\cllh_R)_{\alpha+1}\mathfrak{F}$. It suffices to prove that every finitely generated subgroup of $G$ belongs to $\widehat{\cllh_R}\mathfrak{F}$, and so we may assume that $G$ is finitely generated. This implies $G/H$ is finitely generated, so $G/H \in \clh_R(\cllh_R)_\alpha\mathfrak{F}$. Take a finite length signed permutation resolution of $G/H$ with stabilisers in $(\cllh_R)_\alpha\mathfrak{F}$. Restricting this resolution to $G$ gives a finite length signed permutation resolution whose stabilisers are extensions of $H$ by stabilisers in the original resolution of $G/H$, so the stabilisers are in $(\widehat{\cllh_R}\mathfrak{F})((\cllh_R)_\alpha\mathfrak{F}) \leq \widehat{\cllh_R}\mathfrak{F}$, where the inequality holds by our inductive hypothesis, and hence $G \in \widehat{\cllh_R}\mathfrak{F}$.
\end{enumerate}
\end{proof}

\begin{rem}
The reason we sometimes use $\cll$ rather than $\cll'$ is that $\cll$ is closed under extensions; if one could show the same was true for $\cll'$ then one could construct a class containing all elementary amenable groups (see below) for which the main result would hold. However, we can still recover `most' elementary amenable groups using a combination of $\widehat{\cllh_R}\mathfrak{F}$ and $\widehat{\cll'\clh_R}\mathfrak{F}$.
\end{rem}

We can also compare the classes produced by using different base rings.

\begin{lem}
\label{Ralgebra}
Suppose $S$ is a commutative profinite $R$-algebra. Then $\widehat{\clh_R}\mathfrak{X} \leq \widehat{\clh_S}\mathfrak{X}$, $\widehat{\cllh_R}\mathfrak{X} \leq \widehat{\cllh_S}\mathfrak{X}$ and $\widehat{\cll'\clh_R}\mathfrak{X} \leq \widehat{\cll'\clh_S}\mathfrak{X}$.
\end{lem}
\begin{proof}
Clearly $\mathfrak{X} \leq \mathfrak{Y} \Rightarrow \cll\mathfrak{X} \leq \cll\mathfrak{Y}$ and $\cll'\mathfrak{X} \leq \cll'\mathfrak{Y}$, so we just need to show $\mathfrak{X} \leq \mathfrak{Y} \Rightarrow \clh_R\mathfrak{X} \leq \clh_S\mathfrak{Y}$: then it will follow by induction that for each $\alpha$ that $(\clh_R)_\alpha\mathfrak{X} \leq (\clh_S)_\alpha\mathfrak{X}$, $(\cllh_R)_\alpha\mathfrak{X} \leq (\cllh_S)_\alpha\mathfrak{X}$ and $(\cll'\clh_R)_\alpha\mathfrak{X} \leq (\cll'\clh_S)_\alpha\mathfrak{X}$, as required. Given $G \in \clh_R\mathfrak{X}$ and a finite length signed permutation resolution
\begin{equation*}
0 \rightarrow P_n \rightarrow P_{n-1} \rightarrow \cdots \rightarrow P_0 \rightarrow R \rightarrow 0 \tag{$\ast$}
\end{equation*}
with stabilisers in $\mathfrak{X}$, note that, since every module in the sequence is $R$-free, the sequence is $R$-split, so the sequence
\begin{equation*}
0 \rightarrow S \hat{\otimes}_R P_n \rightarrow S \hat{\otimes}_R P_{n-1} \rightarrow \cdots \rightarrow S \hat{\otimes}_R P_0 \rightarrow S \hat{\otimes}_R R \cong S \rightarrow 0 \tag{$\ast \ast$}
\end{equation*}
is exact -- here each module is made into an $S \llbracket G \rrbracket$-module by taking the $S$-action on $S$ and the $G$-action on $P_i$.

Now, for a signed $R \llbracket G \rrbracket$ permutation module $R \llbracket X \rrbracket$, $S \hat{\otimes}_R R \llbracket X \rrbracket = S \llbracket X \rrbracket$ as $S$-modules by \cite[Proposition 7.7.8]{Wilson}, and then clearly the $G$-action makes this into a signed $S \llbracket G \rrbracket$ permutation module. So, applying this to ($\ast \ast$), we have a finite length signed permutation resolution of $S$ as an $S \llbracket G \rrbracket$-module, and the stabilisers are all in $\mathfrak{X}$ because the stabilisers in ($\ast$) are, so we are done.
\end{proof}

The next lemma gives a profinite analogue of the Eilenberg swindle; it is very similar to \cite[Exercise 11.7.3(a)]{Wilson}, though using a slightly different definition of free modules. Recall that projective modules are summands of free ones.

\begin{lem}
Suppose $P \in PMod(\Lambda)$ is projective, where $\Lambda$ is a profinite $R$-algebra. Then there is a free $F \in PMod(\Lambda)$ such that $P \oplus F$ is free.
\end{lem}
\begin{proof}
Take $Q \in PMod(\Lambda)$ projective such that $P \oplus Q$ is free on some space $X$. Take $F$ to be a countably infinite direct sum of copies of $Q \oplus P$: by the universal properties of coproducts and free modules, $F$ is free on the (profinite completion of the) countably infinite disjoint union of copies of $X$. So is $$P \oplus F = P \oplus Q \oplus P \oplus Q \oplus \cdots,$$ for the same reason.
\end{proof}

Note that, in the same way, for a summand $P$ of a signed permutation module in $PMod(R \llbracket G \rrbracket)$ there is a signed permutation module $F$ such that $P \oplus F$ is a signed permutation module. It is this trick that allows us to define $\clh_R$ using finite length resolutions of signed permutation modules, rather than resolutions of summands of signed permutation modules, without losing anything: we can always replace a resolution of the latter kind with one of the former. In particular we get the following corollary.

We define the cohomological dimension of a profinite group $G$ over $R$, $\cd_RG$, to be $\pd_{R \llbracket G \rrbracket}R$, where $R$ has trivial $G$-action.

\begin{cor}
\label{fincdinhier}
Groups of finite cohomological dimension over $R$ are in $\clh_R\mathfrak{I}$.
\end{cor}
\begin{proof}
Put $\Lambda = R \llbracket G \rrbracket$. Given a finite length projective resolution of $R$, $$0 \rightarrow P_n \rightarrow P_{n-1} \rightarrow \cdots \rightarrow P_0 \rightarrow R \rightarrow 0,$$ we can assume $P_0,\ldots,P_{n-1}$ are free. Indeed, one can see this inductively: if $P_0,\ldots,P_{i-1}$ are free, $i \leq n-1$, take some $Q$ such that $P_i \oplus Q$ is free, and replace $P_i,P_{i+1}$ with $P_i \oplus Q,P_{i+1} \oplus Q$, with the map between them given by $(P_{i+1} \rightarrow P_i)\oplus \id_Q$. Then take a free module $F$ such that $P_n \oplus F$ is free: $$0 \rightarrow P_n \oplus F \rightarrow P_{n-1} \oplus F \rightarrow P_{n-2} \rightarrow \cdots \rightarrow P_0 \rightarrow R \rightarrow 0$$ gives the required resolution.
\end{proof}

We now define a class of profinite groups: the elementary amenable profinite groups. The definition is entirely analogous to the hereditary definition of elementary amenable abstract groups given in \cite{H-L}. Let $\mathscr{X}_0 = \mathfrak{I}$, and let $\mathscr{X}_1$ be the class of profinite groups which are finitely generated abelian by finite. Now define $\mathscr{X}_\alpha = (\cll\mathscr{X}_{\alpha-1})\mathscr{X}_1$ for $\alpha$ a successor ordinal, and for $\alpha$ a limit define $\mathscr{X}_\alpha = \bigcup_{\beta < \alpha} \mathscr{X}_\beta$. Then $\mathscr{X} = \bigcup_\alpha \mathscr{X}_\alpha$ is the class of elementary amenable profinite groups. For $G \in \mathscr{X}$ we define the \emph{class} of $G$ to be the least $\alpha$ with $G \in \mathscr{X}_\alpha$.

Note that soluble profinite groups are clearly elementary amenable.

Let $\pi$ be a finite set of primes.

\begin{prop}
\label{elemamen}
Elementary amenable pro-$\pi$ groups are in $\widehat{\cllh_{\hat{\mathbb{Z}}}}\mathfrak{F}$.
\end{prop}
\begin{proof}
We use induction on the elementary amenable class $\alpha$; the case $\alpha=0$ is trivial. The case of limit ordinals is also trivial. So suppose $\alpha$ is a successor, and suppose $\mathscr{X}_{\alpha-1} \leq \widehat{\cllh_{\hat{\mathbb{Z}}}}\mathfrak{F}$. Then $\cll\mathscr{X}_{\alpha-1} \leq \cll\widehat{\cllh_{\hat{\mathbb{Z}}}}\mathfrak{F} \leq \cllh_{\hat{\mathbb{Z}}}\widehat{\cllh_{\hat{\mathbb{Z}}}}\mathfrak{F} = \widehat{\cllh_{\hat{\mathbb{Z}}}}\mathfrak{F}$. Suppose $G \in \mathscr{X}_\alpha$, and take a normal subgroup $G_1 \in \cll\mathscr{X}_{\alpha-1}$ such that $G/G_1$ is in $\mathscr{X}_1$. Now $G/G_1$ is virtually torsion-free finitely generated abelian, so it has a finite index subgroup which is torsion-free abelian and hence this subgroup has finite cohomological dimension by \cite[Proposition 8.2.1, Theorem 11.6.9]{Wilson}. Therefore by Corollary \ref{fincdinhier} it is in $\clh_{\hat{\mathbb{Z}}}\mathfrak{I} \leq \clh_{\hat{\mathbb{Z}}}\mathfrak{F}$, and hence $G/G_1$ is in $\clh_{\hat{\mathbb{Z}}}\mathfrak{F}$ too by Proposition \ref{subext+}(ii). Therefore $G \in \widehat{\cllh_{\hat{\mathbb{Z}}}}\mathfrak{F}$ by Proposition \ref{subext}(iii).
\end{proof}

Now we note that many elementary amenable profinite groups are prosoluble: these include soluble profinite groups, and by the Feit-Thompson theorem they include all elementary amenable pro-$2'$ groups, where $2'$ is the set of all primes but $2$.

\begin{cor}
\label{prosolea}
Elementary amenable prosoluble groups are in $\widehat{\cll'\clh_{\hat{\mathbb{Z}}}}\mathfrak{F}$.
\end{cor}
\begin{proof}
We show that these groups are in $\cll'\widehat{\cllh_{\hat{\mathbb{Z}}}}\mathfrak{F}$. By \cite[Proposition 2.3.9]{R-Z}, prosoluble groups $G$ have a Sylow basis; that is, a choice $\{S_p:p \text{ prime}\}$ of one Sylow subgroup for each $p$ such that $S_pS_q=S_qS_p$ for each $p,q$. Therefore, writing $p_n$ for the $n$th prime, we have a subgroup $G_n=S_{p_1} \cdots S_{p_n}$ for each $n$, and hence a direct system $\{G_n\}$ of subgroups of $G$ whose union is dense in $G$. By Proposition \ref{elemamen} each $G_n$ is in $\widehat{\cllh_{\hat{\mathbb{Z}}}}\mathfrak{F}$, so we are done.
\end{proof}

Note that in fact this shows that, for any prosoluble group $G$ -- and hence any profinite group of odd order -- if each $G_n$ is in some $\widehat{\cll'\clh_R}\mathfrak{X}$, in the same notation as above, then $G$ is too.

Profinite groups acting on profinite trees with well-behaved stabilisers give further examples of groups in our class, in the spirit of \cite[2.2(iii)]{Krop}, though the profinite case seems to be rather harder to control here than the abstract one. See \cite[Chapter 9.2]{R-Z} for the definitions of (proper) pro-$\mathcal{C}$ free products with amalgamation.

\begin{lem}
Suppose $G_1,G_2,H$ are pro-$\mathcal{C}$ groups, where $\mathcal{C}$ is a class of finite groups closed under taking subgroups, quotients and extensions. Suppose we have $G_1,G_2,H \in \widehat{\cllh_R}\mathfrak{X}$ (or $\widehat{\cll'\clh_R}\mathfrak{X}$). Write $G_1 \ast_H G_2$ for the free pro-$\mathcal{C}$ product of $G_1$ and $G_2$ with amalgamation by $H$, and suppose it is proper. Then $G_1 \ast_H G_2 \in \widehat{\cllh_R}\mathfrak{X}$ (or $\widehat{\cll'\clh_R}\mathfrak{X}$).
\end{lem}
\begin{proof}
We get a finite length permutation resolution from \cite[Theorem 2.1]{G-R}.
\end{proof}

We finish this section by listing, for convenience, some groups in $\widehat{\cll'\clh_{\hat{\mathbb{Z}}}}\mathfrak{F}$.
\begin{itemize}
\item Finite groups (with the discrete topology) are in $\mathfrak{F}$.
\item Profinite groups of finite virtual cohomological dimension over $\hat{\mathbb{Z}}$ are in $\clh_{\hat{\mathbb{Z}}}\mathfrak{F}$, by Corollary \ref{fincdinhier} and Proposition \ref{subext}(ii). Hence:
\item Free profinite groups are in $\clh_{\hat{\mathbb{Z}}}\mathfrak{F}$.
\item Soluble profinite groups are in $(\cll'\clh_{\hat{\mathbb{Z}}})_\omega\mathfrak{F}$, by Corollary \ref{prosolea}.
\item Elementary amenable pro-$p$ groups are in $\widehat{\cll'\clh_{\hat{\mathbb{Z}}}}\mathfrak{F}$ for all $p$, by Proposition \ref{elemamen}, and elementary amenable profinite groups of odd order are too, by Corollary \ref{prosolea}.
\end{itemize}

Finally, for $G$ a compact $p$-adic analytic group, $G$ is a virtual Poincar\'{e} duality group at the prime $p$ by \cite[Theorem 5.1.9]{S-W}, and hence by definition $G$ has finite virtual cohomological dimension over $\mathbb{Z}_p$, and so $G \in \clh_{\mathbb{Z}_p}\mathfrak{F}$. In particular this includes $\mathbb{Z}_p$-linear groups by \cite[Proposition 8.5.1]{Wilson}.

\section{Type L Systems}
\label{typeIsys}

Recall that a $\Lambda$-module $M$ is said to be \emph{of type $\FP_n$ over $\Lambda$} if it has a projective resolution $P_\ast$ for which $P_0, \ldots, P_n$ are finitely generated; it is said to be of type $\FP_\infty$ over $\Lambda$ if it of type $\FP_n$ over $\Lambda$ for all $n$. This is equivalent to having a projective resolution with $P_n$ finitely generated for all $n$, by \cite[Proposition 1.5]{Bieri}. When the choice of $\Lambda$ is clear, we will just say $M$ is of type $\FP_n$.

We assume a familiarity with the definitions of $\Ext$ functors with both variables profinite: given $M,N \in PMod(\Lambda)$, we define $\Ext_\Lambda^n(M,N)$ by $H^n(\Hom_\Lambda(P_\ast,N))$, where $P_\ast$ is a projective resolution of $M$. This makes each $\Ext_\Lambda^n$ a functor of both variables, and we get the usual long exact sequences. To establish a notational convention: when the first variable $M$ is specified to be a $\Lambda$-module of type $\FP_\infty$, we can and will always take $\Ext_\Lambda^n$ to be a functor $PMod(\Lambda) \times PMod(\Lambda) \rightarrow PMod(R)$ for each $n$. This is because we can take each $P_n$ finitely generated and then $\Hom_\Lambda(P_n,N)$ has a natural profinite topology, as in \cite[Section 3.7]{S-W}. When we want our $\Ext$ groups to be abstract $U(R)$-modules, we will compose this with the forgetful functor $U: PMod(R) \rightarrow Mod(U(R))$ which forgets the topology. On the other hand, when we just know that $M \in PMod(\Lambda)$, each $\Ext_\Lambda^n$ will be thought of as a functor $PMod(\Lambda) \times PMod(\Lambda) \rightarrow Mod(U(R))$.

To be able to use the hierarchy of groups defined in the last section, we want to relate the construction of a group within the hierarchy to its cohomology, and so gain results about the structure of the group, analogously to \cite{Krop}. Specifically, this section will deal with the interaction of cohomology and the closure operation $\cll$, and the next one with the interaction of cohomology and $\clh_R$.

See \cite[Definition 2.6.13]{Weibel} for the definition of direct systems and direct limits. They are exact in categories of abstract modules.

Let $R$ be a commutative profinite ring and $\Lambda$ a profinite $R$-algebra. We call a direct system $\{A^i:i \in I\}$ of $\Lambda$-modules a Type L system if there is some $i_0 \in I$ such that the maps $f^{i_0i}: A^{i_0} \rightarrow A^i$ for each $i \geq i_0$ are all epimorphisms. Then, considering $\{A^i\}$ as a direct system in $TMod(\Lambda)$, $$U(\varinjlim_{TMod(\Lambda)}A^i) = \varinjlim_{Mod(U(\Lambda))} U(A^i) = U(A^{i_0})/\bigcup_{i \geq i_0} \ker(Uf^{i_0i})$$ by \cite[Lemma 2.6.14]{Weibel} and the remark after Proposition \ref{limandcolim}. We can see from this, and from the proof of Proposition \ref{limandcolim}, that $\varinjlim_{TMod(\Lambda)}A^i$ has as its underlying module $U(A^{i_0})/\bigcup_{i \geq i_0} \ker(Uf^{i_0i})$, with the strongest topology making each map $$f^i: A^i \rightarrow U(A^{i_0})/\bigcup_{i \geq i_0} \ker(Uf^{i_0i})$$ continuous, such that $U(A^{i_0})/\bigcup_{i \geq i_0} \ker(Uf^{i_0i})$ is made into a topological $\Lambda$-module. But the quotient topology induced by the map $f^{i_0}$ satisfies these conditions: it makes $U(A^{i_0})/\bigcup_{i \geq i_0} \ker(Uf^{i_0i})$ into a topological $\Lambda$-module by \cite[III.6.6]{Bourbaki}; it makes $f^{i_0}$ continuous; it makes each $f^i$, $i \geq i_0$ continuous because, given an open set $U$ in $A^{i_0}/\bigcup ker(f^{i_0i})$, $$(f^{i_0})^{-1}(U) = (f^{i_0i})^{-1}(f^{i})^{-1}(U)$$ is open in $A^{i_0}$, and $A^i$ has the quotient topology coming from $f^{i_0i}$ (because all the modules are compact and Hausdorff), and by the definition of the quotient topology this means that $(f^{i})^{-1}(U)$ is open in $A^i$, as required. Hence $$\varinjlim_{TMod(\Lambda)}A^i \cong A^{i_0}/\bigcup_{i \geq i_0} \ker(f^{i_0i})$$ as topological modules. Note that this quotient is compact, as the continuous image of $A^{i_0}$.

Recall from Corollary \ref{pccomp} that we know $\varinjlim_{PMod(\Lambda)}A^i$ is the profinite completion of $\varinjlim_{TMod(\Lambda)}A^i$. Hence there is a canonical homomorphism $$\phi: \varinjlim_{TMod(\Lambda)}A^i \rightarrow \varinjlim_{PMod(\Lambda)}A^i.$$ Since $\varinjlim_{PMod(\Lambda)}A^i$ is Hausdorff, $\ker(\phi)=\phi^{-1}(0)$ is closed in $\varinjlim_{TMod(\Lambda)}A^i$; in particular, $\ker(\phi)$ contains the closure of $\{0\}$ in $\varinjlim_{TMod(\Lambda)}A^i$. Hence $\phi$ factors (uniquely) as $$\varinjlim_{TMod(\Lambda)}A^i \cong A^{i_0}/\bigcup_{i \geq i_0} \ker(f^{i_0i}) \xrightarrow{\psi} A^{i_0}/\overline{\bigcup_{i \geq i_0} \ker(f^{i_0i})} \rightarrow \varinjlim_{PMod(\Lambda)}A^i.$$ Now $A^{i_0}/\overline{\bigcup_{i \geq i_0} \ker(f^{i_0i})}$ is a quotient of a profinite $\Lambda$-module by a closed submodule, so it is profinite; hence, by the universal property of profinite completions, $\psi$ factors uniquely through $\phi$. It follows that $\varinjlim_{PMod(\Lambda)}A^i \cong A^{i_0}/\overline{\bigcup_{i \geq i_0} \ker(f^{i_0i})}$.

Given $A \in PMod(\Lambda)$, we can think of $A$ as an object of $PMod(R)$ by restriction. This functor is representable in the sense that it is given by $\Hom_{\Lambda}(\Lambda,-): PMod(\Lambda) \rightarrow PMod(R)$.

\begin{lem}
\label{typeIres}
Direct limits of Type L systems commute with restriction. Explicitly, let $\{A^i\}$ be a Type L system in $PMod(\Lambda)$. Then $$\varinjlim_{PMod(R)} \Hom_{\Lambda}(\Lambda, A^i) = \Hom_{\Lambda}(\Lambda, \varinjlim_{PMod(\Lambda)} A^i).$$
\end{lem}
\begin{proof}
By our construction of direct limits in $PMod(\Lambda)$, both sides are just the restriction to $PMod(R)$ of $A^{i_0}/\overline{\bigcup_{i \geq i_0} \ker(f^{i_0i})}$, given the quotient topology.
\end{proof}

It follows by additivity that $$\varinjlim_{PMod(R)} \Hom_{\Lambda}(P, A^i) = \Hom_{\Lambda}(P, \varinjlim_{PMod(\Lambda)} A^i)$$ for all finitely generated projective $P \in PMod(\Lambda)$.

\begin{prop}
\label{epitypeI}
Suppose $M \in PMod(\Lambda)$ is of type $\FP_\infty$, and let $\{A^i\}$ be a Type L system in $PMod(\Lambda)$. Then for each $n$ we have an epimorphism $$\varinjlim_{PMod(R)} \Ext_\Lambda^n(M,A^i) \rightarrow \Ext_\Lambda^n(M, \varinjlim_{PMod(\Lambda)} A^i).$$
\end{prop}
\begin{proof}
We show this in two stages. Take a projective resolution $$\cdots \rightarrow P_2 \xrightarrow{f_1} P_1 \xrightarrow{f_0} P_0 \rightarrow 0$$ of $M$ with each $P_n$ finitely generated. We will show first that $$H^n(\varinjlim_{PMod(R)} \Hom_\Lambda(P_\ast,A^i)) = \Ext_\Lambda^n(M, \varinjlim_{PMod(\Lambda)} A^i).$$ To see this, consider the commutative diagram
\[
\xymatrix{0 \ar[r] & \varinjlim_{PMod(R)} \Hom_\Lambda(P_0,A^i) \ar[r] \ar[d] & \varinjlim_{PMod(R)} \Hom_\Lambda(P_1,A^i) \ar[r] \ar[d] & \cdots \\
0 \ar[r] & \Hom_{\Lambda}(P_0, \varinjlim_{PMod(\Lambda)} A^i) \ar[r] & \Hom_{\Lambda}(P_1, \varinjlim_{PMod(\Lambda)} A^i) \ar[r] & \cdots}
\]
in $PMod(R)$. The homology of the top row is $$H^n(\varinjlim_{PMod(R)} \Hom_\Lambda(P_\ast,A^i)),$$ the homology of the bottom row is $\Ext_\Lambda^n(M, \varinjlim_{PMod(\Lambda)} A^i)$, and the previous lemma shows that the vertical maps are all isomorphisms.

The second stage is to give epimorphisms $$\varinjlim_{PMod(R)} \Ext_\Lambda^n(M,A^i) \rightarrow H^n(\varinjlim_{PMod(R)} \Hom_\Lambda(P_\ast,A^i)).$$ Recall that $\varinjlim_{PMod(R)}$ is right-exact, so that we get an exact sequence
\begin{align*}
\varinjlim_{PMod(R)}\ker(\Hom_\Lambda(f_n,A^i)) &\rightarrow \varinjlim_{PMod(R)} \Hom_\Lambda(P_n,A^i) \\
&\rightarrow \varinjlim_{PMod(R)} \Hom_\Lambda(P_{n+1},A^i),
\end{align*}
and hence an epimorphism $$\varinjlim_{PMod(R)}\ker(\Hom_\Lambda(f_n,A^i)) \rightarrow \ker(\varinjlim_{PMod(R)} \Hom_\Lambda(f_n,A^i)).$$ Now consider the commutative diagram
\[
\xymatrix{\varinjlim_{PMod(R)} \Hom_\Lambda(P_{n-1},A^i) \ar[r] \ar[d]^\cong & \varinjlim_{PMod(R)}\ker(\Hom_\Lambda(f_n,A^i)) \ar@{->>}[d] \\
\varinjlim_{PMod(R)} \Hom_\Lambda(P_{n-1},A^i) \ar[r] & \ker(\varinjlim_{PMod(R)} \Hom_\Lambda(f_n,A^i)) \\
\ar[r] & \varinjlim_{PMod(R)} \Ext_\Lambda^n(M,A^i) \ar[r] \ar[d] & 0 \\
\ar[r] & H^n(\varinjlim_{PMod(R)} \Hom_\Lambda(P_\ast,A^i)) \ar[r] & 0}
\]
whose top row is exact because $\varinjlim_{PMod(R)}$ is right-exact, and whose bottom row is exact by definition of homology. It follows by the Five Lemma that $$\varinjlim_{PMod(R)} \Ext_\Lambda^n(M,A^i) \rightarrow H^n(\varinjlim_{PMod(R)} \Hom_\Lambda(P_\ast,A^i))$$ is an epimorphism, as required.
\end{proof}

The next lemma will allow us to make new Type L systems from old ones.

\begin{lem}
\label{typeI+}
Suppose $G \in PGrp$. Suppose $M \in PMod(R \llbracket G \rrbracket)$ is projective as an $R$-module, by restriction, and let $\{A^i\}$ be a Type L system in $PMod(R \llbracket G \rrbracket)$. Then $\{M \hat{\otimes}_R A^i\}$ is a Type L system in $PMod(R \llbracket G \rrbracket)$, where each $M \hat{\otimes}_R A^i$ is given the diagonal $G$-action.
\end{lem}
\begin{proof}
Because $M \hat{\otimes}_R -$ preserves epimorphisms, we just need to show that it commutes with direct limits of Type L systems; that is, we have to show that if $\{A^i,f^{ij}\}$ is a Type L system, then $$M \hat{\otimes}_R \varinjlim_{PMod(R)} A^i \cong \varinjlim_{PMod(R)} (M \hat{\otimes}_R A^i).$$ We have a canonical homomorphism $$g: \varinjlim_{PMod(R)} (M \hat{\otimes}_R A^i) \rightarrow M \hat{\otimes}_R \varinjlim_{PMod(R)} A^i;$$ it is an epimorphism because the epimorphism $$g^{i_0}: M \hat{\otimes}_R A^{i_0} \rightarrow M \hat{\otimes}_R \varinjlim_{PMod(R)} A^i$$ factors as $$M \hat{\otimes}_R A^{i_0} \xrightarrow{h} \varinjlim_{PMod(R)} (M \hat{\otimes}_R A^i) \xrightarrow{g} M \hat{\otimes}_R \varinjlim_{PMod(R)} A^i.$$ In fact we will show that $\ker g^{i_0} \subseteq \ker h$; this implies that $g$ is injective, as required.

Now $$\ker g^{i_0} = M \hat{\otimes}_R \ker f^{i_0} = M \hat{\otimes}_R \overline{\bigcup_{i \geq i_0} \ker f^{i_0i}}$$ by the exactness of $M \hat{\otimes}_R -$ (because $M$ is $R$-projective) and the construction of direct limits of Type L systems; moreover, writing $g^{i_0i}$ for $M \hat{\otimes}_R A^{i_0} \rightarrow M \hat{\otimes}_R A^i$, we have by exactness of $M \hat{\otimes}_R -$ again that $\ker g^{i_0i} = M \hat{\otimes}_R \ker f^{i_0i}$. Hence, by the construction of direct limits of Type L systems, $$\ker h = \overline{\bigcup_{i \geq i_0} \ker g^{i_0i}} = \overline{\bigcup_{i \geq i_0} M \hat{\otimes}_R \ker f^{i_0i}}.$$ Thus we are reduced to showing that the subspace $$\bigcup_{i \geq i_0} M \hat{\otimes}_R \ker f^{i_0i} \subseteq M \hat{\otimes}_R \overline{\bigcup_{i \geq i_0} \ker f^{i_0i}}$$ is dense. This can be seen by considering inverse limits: if $M = \varprojlim M_j$, then $\overline{\bigcup_{i \geq i_0} \ker f^{i_0i}} = \varprojlim N_k$, with all the $M_j$, $N_k$ finite, then $$M \hat{\otimes}_R \overline{\bigcup_{i \geq i_0} \ker f^{i_0i}} = \varprojlim M_j \hat{\otimes}_R N_k,$$ and by the denseness of $\bigcup_{i \geq i_0} \ker f^{i_0i}$ in $\overline{\bigcup_{i \geq i_0} \ker f^{i_0i}}$, for each $k$ there is some $i$ such that $\ker f^{i_0i} \rightarrow N_k$ is surjective, so $M_j \hat{\otimes}_R \ker f^{i_0i} \rightarrow M_j \hat{\otimes}_R N_k$ is too. Denseness follows by \cite[Lemma 1.1.7]{R-Z}.
\end{proof}

Finally, we need one more result to apply this to the problem of getting information about group structure. Suppose $G \in PGrp$, let $H$ be a subgroup of $G$, and let $\{H_i\}$ be a direct system of (closed) subgroups of $H$, with inclusion maps between them, whose union $H'$ is dense in $H$ -- note that $H'$ is an (abstract) subgroup of $H$, because the system is direct. Thus we get a corresponding direct system $\{R \llbracket G/H_i \rrbracket\}$ of $R \llbracket G \rrbracket$ permutation modules whose maps come from quotients $G/H_i \rightarrow G/H_j$. Note that this system is Type L, because the maps $R \llbracket G/H_i \rrbracket \rightarrow R \llbracket G/H_j \rrbracket$ are all epimorphisms.

\begin{lem}
\label{prodirlim}
$\varinjlim_{PMod(R \llbracket G \rrbracket)} R \llbracket G/H_i \rrbracket = R \llbracket G/H \rrbracket$. Hence $\{R \llbracket G/H_i \rrbracket\}$ is Type L.
\end{lem}
\begin{proof}
Recall that, for $X \in G \text{-} Pro$ and $M \in PMod(R \llbracket G \rrbracket)$, we let $C_G(X,M)$ be the $U(R)$-module of continuous $G$-maps $X \rightarrow M$. We have
\begin{align*}
\Hom_{R \llbracket G \rrbracket}(\varinjlim_{PMod(R \llbracket G \rrbracket)} R \llbracket G/H_i \rrbracket,M) &= \varprojlim_{Mod(U(R))} \Hom_{R \llbracket G \rrbracket}(R \llbracket G/H_i \rrbracket,M) \\
&= \varprojlim_{Mod(U(R))} C_G(G/H_i,M) \\
&= C_G(\varinjlim_{G \text{-} Pro} G/H_i,M) \\
&= \Hom_{R \llbracket G \rrbracket}(R \llbracket \varinjlim_{G \text{-} Pro} G/H_i \rrbracket,M):
\end{align*}
the first and third equalities hold by the universal property of colimits; the second and fourth hold by the universal property of permutation modules, Lemma \ref{C_G}. Since this holds for all $M$, we have $\varinjlim_{PMod(R \llbracket G \rrbracket)} R \llbracket G/H_i \rrbracket = R \llbracket \varinjlim_{G \text{-} Pro} G/H_i \rrbracket$, so we just need to show that $\varinjlim_{G \text{-} Pro} G/H_i = G/H$.

To see this, we will show first that $\varinjlim_{G \text{-} Top} G/H_i = G/H'$. Note that we have compatible epimorphisms $G/H_i \rightarrow G/H'$, and hence an epimorphism $f: \varinjlim_{G \text{-} Top} G/H_i \rightarrow G/H'$. Note also that the maps $G/H_i \rightarrow \varinjlim_{G \text{-} Top} G/H_i$ are surjective. Suppose $f(x)=f(y)$ for $x,y \in \varinjlim_{G \text{-} Top} G/H_i$. Take a representative $x'$ of $x$ in some $G/H_{i_1}$, and a representative $y'$ of $y$ in some $G/H_{i_2}$. Now the images of $x'$ and $y'$ are in the same left coset of $H'$ in $G$, i.e. $x'h_1 = y'h_2$ for some $h_1,h_2 \in H'$. Write $H_j$ for the subgroup of $H'$ generated by $h_1, h_2, x$ and $y$. Thus the images of $x'$ and $y'$ are in the same left coset of $H_j$ in $G$, i.e. $x'$ and $y'$ have the same image in $G/H_j$ and hence in $\varinjlim_{G \text{-} Top} G/H_i$, so $x=y$, and $f$ is injective. Finally, note that, in exactly the same way as the construction of Type L direct limits, $\varinjlim_{G \text{-} Top} G/H_i$ has the quotient topology coming from $G$, which is the same as the one on $G/H'$. Thus, by Corollary \ref{pccomp}, $\varinjlim_{G \text{-} Pro} G/H_i$ is the profinite completion of $G/H'$, which is just $G/\overline{H'} = G/H$ by the same argument as for Type L systems.
\end{proof}

By Lemma \ref{typeIres}, $\varinjlim_{PMod(R)} R \llbracket G/H_i \rrbracket = R \llbracket G/H \rrbracket$ as well; indeed, by the same lemma, any compatible collection of $G$-actions on these modules gives a direct limit whose underlying $R$-module is $R \llbracket G/H \rrbracket$, and whose $G$-action is just the one coming from any of the quotient maps $$f_i: R \llbracket G/H_i \rrbracket \rightarrow R \llbracket G/H \rrbracket.$$ So if $R \llbracket G/H;\sigma \rrbracket$ is a signed $R \llbracket G \rrbracket$ permutation module, define $R \llbracket G/H_i;\sigma_i \rrbracket$ for each $i$ to be a signed $R \llbracket G \rrbracket$ permutation module by the $G$-action $\sigma_i(g,x) = gx$ if $\sigma(g,f_i(x)) = gf_i(x)$ and $\sigma_i(g,x) = -gx$ if $\sigma(g,f_i(x)) = -gf_i(x)$, for all $g \in G, x \in G/H_i \cup -G/H_i$. Clearly these $G$-actions are all compatible, and they have as their direct limit (in $PMod(\Lambda)$) $R \llbracket G/H;\sigma \rrbracket$. In particular, we get the following result.

\begin{cor}
\label{prodirlim+}
Given a signed $R \llbracket G \rrbracket$ permutation module $R \llbracket G/H;\sigma \rrbracket$, and a direct system $\{H_i\}$ of subgroups of $H$ whose union is dense in $H$, there is a Type L system of signed permutation modules of the form $R \llbracket G/H_i;\sigma_i \rrbracket$ whose direct limit is $R \llbracket G/H;\sigma \rrbracket$.
\end{cor}

\section{Type H Systems}
\label{typeIIsys}

As before, let $R$ be a commutative profinite ring, $\Lambda$ a profinite $R$-algebra, and $M \in PMod(\Lambda)$ of type $\FP_\infty$. Suppose $A \in PMod(\Lambda)$ has the form $\varprojlim_{j \in J} A_j$, where each $A_j \in PMod(\Lambda)$ is finite. Suppose in addition that each $A_j$ is a direct sum $A_{j,1} \oplus \cdots \oplus A_{j,n_j}$ of $\Lambda$-modules such that, whenever $j_1 \geq j_2$, the morphism $\phi_{j_1j_2}: A_{j_1} \rightarrow A_{j_2}$ has the property that, for each $k$, $\phi_{j_1j_2}(A_{j_1,k})$ is contained in some $A_{j_2,k'}$. Then we say $A$ has the structure of a Type H system.

In the same notation, write $I_j$ for the set $\{1,\ldots,n_j\}$. Then the structure of the Type H system induces a map $\psi_{j_1j_2}: I_{j_1} \rightarrow I_{j_2}$ for each $j_1 \geq j_2$ in $J$, giving an inverse system $\{I_j:j \in J\}$: if $\phi_{j_1j_2}(A_{j_1,k}) \subseteq A_{j_2,k'}$, define $\psi_{j_1j_2}(k)=k'$. Write $I$ for the inverse limit and $\iota_j$ for the map $I \rightarrow I_j$. $I$ is clearly profinite, because it is the inverse limit of a system of finite sets; also $I$ is non-empty by \cite[Proposition 1.1.4]{R-Z}. Now pick $i \in I$. We call $A^i = \varprojlim_j A_{j,\iota_j(i)}$ the $i$th component of $A$.

\begin{prop}
\label{epitypeII}
Suppose $A \in PMod(\Lambda)$ has the structure of a Type H system. Suppose $\{A^i: i \in I\}$ are the components of $A$. Then for each $n$ we have an epimorphism $$\bigoplus_{PMod(R),i} \Ext_\Lambda^n(M,A^i) \rightarrow \Ext_\Lambda^n(M,A).$$
\end{prop}
\begin{proof}
Many aspects of the Type H structure carry over to $\Ext_\Lambda^n(M,A)$. $$\Ext_\Lambda^n(M,A) = \varprojlim_{PMod(R),j} \Ext_\Lambda^n(M,A_j)$$ by \cite[Theorem 3.7.2]{S-W}, and similarly
\begin{equation*}
\Ext_\Lambda^n(M,A^i) = \varprojlim_{PMod(R),j} \Ext_\Lambda^n(M,A_{j,\iota_j(i)}) \tag{$\ast$}
\end{equation*}
for each $i \in I$. By additivity, $$\Ext_\Lambda^n(M,A_j) = \Ext_\Lambda^n(M,A_{j,1}) \oplus \cdots \oplus \Ext_\Lambda^n(M,A_{j,n_j}).$$ Note that each $\Ext_\Lambda^n(M,A_j)$ and $\Ext_\Lambda^n(M,A_{j,k})$ is finite, because $M$ is of type $\FP_\infty$, and there are only finitely many homomorphisms from a finitely generated $\Lambda$-module to a finite one.

Write $C_j$ for the image of $$g_j: \Ext_\Lambda^n(M,A) \rightarrow \Ext_\Lambda^n(M,A_j):$$ then by \cite[Corollary 1.1.8(a)]{R-Z} we know $\Ext_\Lambda^n(M,A) = \varprojlim_{PMod(R),j} C_j$. We claim that $$f_j: \bigoplus_{PMod(R),i} \Ext_\Lambda^n(M,A^i) \rightarrow C_j$$ is an epimorphism for each $j$, and then the proposition will follow by \cite[Corollary 1.1.6]{R-Z}. To see this claim, fix some $j$, and suppose that the image of $f_j$ is some submodule $C'_j \neq C_j$. We will obtain a contradiction by showing that the image of $f_j$ is strictly larger than $C'_j$.

Now define, for $j' \geq j$, $I'_{j'} \subseteq I_{j'}$ to be those elements $k$ of $I_{j'}$ for which the image of $\Ext_\Lambda^n(M,A_{j',k})$ in $\Ext_\Lambda^n(M,A_j)$ is not contained in $C'_j$. For each $j' \geq j$ the map $g_j$ factors as $$\Ext_\Lambda^n(M,A) \xrightarrow{g_{j'}} \Ext_\Lambda^n(M,A_{j'}) \xrightarrow{g_{j'j}} \Ext_\Lambda^n(M,A_j),$$ so $\im(g_{j'j}) \supseteq \im(g_{j'}) = C_j$; hence $I'_{j'} \neq \varnothing$, and so $I' = \varprojlim_{j' \geq j} I'_{j'} \neq \varnothing$ by \cite[Proposition 1.1.4]{R-Z}.

Pick $i \in I'$. By definition of $I'$, $\Ext_\Lambda^n(M,A_{j,\iota_j(i)})$ is not contained in $C'_j$, so $\Ext_\Lambda^n(M,A_{j,\iota_j(i)}) \setminus C'_j \neq \varnothing$. Suppose that, for each $x \in \Ext_\Lambda^n(M,A_{j,\iota_j(i)}) \setminus C'_j$, there is some $j_x \geq j$ such that $$x \notin \im(f_{j_x,\iota_{j_x}(i)}: \Ext_\Lambda^n(M,A_{j_x,\iota_{j_x}(i)}) \rightarrow \Ext_\Lambda^n(M,A_j)).$$ Since $J$ is directed, there is some $j_0 \in J$ such that $j_0 \geq j_x$ for all $x$ in the finite set $\Ext_\Lambda^n(M,A_{j,\iota_j(i)}) \setminus C'_j$. For each such $x$, $\im(f_{j_0,\iota_{j_0}(i)}) \subseteq \im(f_{j_x,\iota_{j_x}(i)})$ and hence $x \notin \im(f_{j_0,\iota_{j_0}(i)}),$ so that $\im(f_{j_0,\iota_{j_0}(i)}) \subseteq C'_j$. But we chose $\iota_{j_0}(i)$ to be in $I'_{j_0}$, so $\im(f_{j_0,\iota_{j_0}(i)}) \nsubseteq C'_j$, contradicting our supposition. Therefore there must be some $x \in \Ext_\Lambda^n(M,A_{j,\iota_j(i)}) \setminus C'_j$ such that, for every $j' \geq j$, $x \in \im(f_{j',\iota_{j'}}).$

Write $f^i_j$ for the map $\Ext_\Lambda^n(M,A^i) \rightarrow \Ext_\Lambda^n(M,A_j)$, so that by ($\ast$) we have $f^i_j = \varprojlim_{j'} f_{j',\iota_{j'}}$. For every $j' \geq j$ we have $f_{j',\iota_{j'}}^{-1}(x) \neq \varnothing$, and hence, taking inverse limits over $j'$, we get $(f^i_j)^{-1}(x) \neq \varnothing$ by \cite[Proposition 1.1.4]{R-Z}, so that $x \in \im(f^i_j) \setminus C'_j$. Finally, it is clear from the definitions that $\im(f_j) \supseteq \im(f^i_j)$, so $x \in \im(f_j) \setminus C'_j$, proving our claim and giving the result.
\end{proof}

As in the last section, we want to be able to make new Type H systems from old ones.

\begin{lem}
\label{typeII+}
Suppose $G \in PGrp$. Suppose $M,A \in PMod(R \llbracket G \rrbracket)$, $M = \varprojlim_k M_k$, and let $A = \varprojlim_j A_{j,1} \oplus \cdots \oplus A_{j,n_j}$ have the structure of a Type H system in $PMod(\Lambda)$. Suppose $\{A^i: i \in I\}$ are the components of $A$. Then $M \hat{\otimes}_R A \in PMod(R \llbracket G \rrbracket)$, with the diagonal $G$-action, has the structure of a Type H system given by $\varprojlim_{j,k} (M_k \hat{\otimes}_R A_{j,1}) \oplus \cdots \oplus (M_k \hat{\otimes}_R A_{j,n_j})$ with components $\{M \hat{\otimes}_R A^i\}$.
\end{lem}
\begin{proof}
This is immediate, since $\hat{\otimes}_R$ commutes with $\varprojlim$ and finite direct sums commute with both.
\end{proof}

Again, this section finishes with a couple of lemmas allowing us to get information about group structure.

\begin{lem}
\label{1orbit}
Suppose $G \in PGrp$, and suppose $\{X_j\}$ is an inverse system in $G \text{-} Pro$ with $X = \varprojlim_j X_j$. If each $X_j$ has a single $G$-orbit, so does $X$.
\end{lem}
\begin{proof}
Write $\phi_j$ for the map $X \rightarrow X_j$. For $Y \subseteq X$, define $G \cdot Y = \{gy: g \in G, y \in Y\}$. Pick $x \in X$. Then $\phi_j(G \cdot \{x\}) = G \cdot \phi_j(\{x\}) = X_j$ for each $j$, because each $X_j$ has a single $G$-orbit, so $G \cdot \{x\} = \varprojlim_j X_j = X$ by \cite[Corollary 1.1.8(a)]{R-Z}. Hence the orbit of $x$ is the whole of $X$.
\end{proof}

\begin{lem}
\label{typeIIsperm}
Suppose $G \in PGrp$, and suppose $R \llbracket X \rrbracket \in PMod(R \llbracket G \rrbracket)$ is a signed permutation module. Then $R \llbracket X \rrbracket$ has the structure of a Type H system whose components are signed permutation modules $R \llbracket X^i \rrbracket$, where the $X^i$ are the $G$-orbits of $R \llbracket X \rrbracket$.
\end{lem}
\begin{proof}
By Lemma \ref{algstruct} we can write $X \cup -X = \varprojlim X_j \cup -X_j$, where the $X_j \cup -X_j$ are finite quotients of $X \cup -X$ preserving the algebraic structure. If $R = \varprojlim_l R_l$, $R \llbracket X \rrbracket = \varprojlim_{j,l} R_l[X_j]$, and each $R_l[X_j]$ is a signed $R_l \llbracket G \rrbracket$ permutation module. Now as a $G$-space $\overline{X_j} = (X_j \cup -X_j)/\sim$ is the disjoint union of its orbits $\overline{X_{j,1}}, \ldots, \overline{X_{j,n_j}}$, where $\sim$ is the relation $x \sim -x$, so $X_j \cup -X_j$ is the disjoint union of $G$-spaces $X_{j,1} \cup -X_{j,1}, \ldots, X_{j,n_j} \cup -X_{j,n_j}$. Therefore we get $$R_l[X_j] = R_l[X_{j,1}] \oplus \cdots \oplus R_l[X_{j,n_j}].$$ For $l_1 \geq l_2$ and $j_1 \geq j_2$, write $\phi_{(l_1,j_1)(l_2,j_2)}$ for the map $R_{l_1}[X_{j_1}] \rightarrow R_{l_2}[X_{j_2}]$. For each orbit $\overline{X_{j_1,k_1}}$, $$\phi_{(l_1,j_1)(l_2,j_2)}(X_{j_1,k_1} \cup -X_{j_1,k_1}) = X_{j_2,k_2} \cup -X_{j_2,k_2},$$ where $\overline{X_{j_2,k_2}}$ is the image of $\overline{X_{j_1,k_1}}$ in $\overline{X_{j_2}}$ (since $\overline{X_{j_1,k_1}}$ has only one orbit, so any $G$-map image of it has one orbit too). Therefore $$\phi_{(l_1,j_1)(l_2,j_2)}(R_{l_1}[X_{j_1,k_1}]) = R_{l_2}[X_{j_2,k_2}],$$ and hence we have the structure of a Type H system.

Now write $I_j = \{1,\ldots,n_j\}$, define the maps $I_{j'} \rightarrow I_j$ for $j' \geq j$ coming from the Type H structure, and let $I = \varprojlim_j I_j$, $\iota_j: I \rightarrow I_j$. We give a bijection between $I$ and the set of orbits of $R \llbracket X \rrbracket$. Given a $G$-orbit $X'$ of $R \llbracket X \rrbracket$, any $G$-map image of it has one orbit too, so for each $j,l$ the image of $R \llbracket X' \rrbracket \rightarrow R \llbracket X \rrbracket \rightarrow R_l[X_j]$ must be contained in some $R_l[X_{j,i_j}]$ with one orbit. Define the element $i \in I$ to be the inverse limit over $j$ of $i_j$ and define the map $b: \{\text{orbits of } R \llbracket X \rrbracket\} \rightarrow I$ by $b(X')=i$.

Conversely, given $i \in I$, each $\overline{X_{j,\iota_j(i)}}$ has a single $G$-orbit, so $\overline{X^i} = \varprojlim_j \overline{X_{j,\iota_j(i)}}$ does too, by Lemma \ref{1orbit}. It is easy to see the map $i \mapsto X^i$ is inverse to $b$, giving the result.
\end{proof}

\section{The Main Result}
\label{main}

We can now use these results to get information about groups of type $\FP_\infty$ in $\widehat{\cll'\clh_R}\mathfrak{F}$. As usual, $R$ is a commutative profinite ring. Given abelian categories $C,D$, define a $(-\infty,\infty)$ cohomological functor from $C$ to $D$ to be a sequence of additive functors $T^i: C \rightarrow D$, $i \in \mathbb{Z}$, with natural connecting homomorphisms such that for every short exact sequence $0 \rightarrow L \rightarrow M \rightarrow N \rightarrow 0$ in $C$ we get a long exact sequence $$\cdots \rightarrow T^{n-1}(N) \rightarrow T^n(L) \rightarrow T^n(M) \rightarrow T^n(N) \rightarrow \cdots.$$

We start by giving a (slight) generalisation of \cite[3.1]{Krop}, which holds for all $(-\infty,\infty)$ cohomological functors. The proof is a dimension-shifting argument which goes through entirely unchanged.

\begin{lem}
\label{lesneq0}
Let $T^\ast$ be a $(-\infty,\infty)$ cohomological functor from $C$ to $D$. Let $$0 \rightarrow M_r \rightarrow M_{r-1} \rightarrow \cdots \rightarrow M_0 \rightarrow L \rightarrow 0$$ be an exact sequence in $D$. If $T^i(L) \neq 0$ for some $i$ then $T^{i+j}(M_j) \neq 0$ for some $0 \leq j \leq r$.
\end{lem}

Define, for all $n \geq 0$, $$H_R^n(G,-) = \Ext_{R \llbracket G \rrbracket}^n(R,-).$$ Define $H_R^n(G,-)=0$ for $n<0$. The functors $H_R^\ast(G,-)$ thus defined form a $(-\infty,\infty)$ cohomological functor from $PMod(R \llbracket G \rrbracket)$ to $Mod(U(R))$.

The following theorem corresponds roughly to \cite[3.2]{Krop}.

\begin{thm}
\label{neq0h}
Suppose $G \in \widehat{\cll'\clh_R}\mathfrak{X}$ is of type $\FP_\infty$. Then there is some subgroup $H \leq G$ which is in $\mathfrak{X}$, some signed $R \llbracket G \rrbracket$ permutation module $R \llbracket G/H;\sigma \rrbracket$ and some $n$ such that $H_R^n(G,R \llbracket G/H;\sigma \rrbracket) \neq 0$.
\end{thm}
\begin{proof}
Note first that $H_R^0(G,R) = R \neq 0$.

Consider the collection $\mathcal{O}$ of ordinals $\beta$ for which there exists $i \geq 0$ and $H \leq G$ such that $H \in (\cll'\clh_R)_\beta \mathfrak{X}$ and $H_R^i(G,R \llbracket G/H;\tau \rrbracket) \neq 0$, for some signed $R \llbracket G \rrbracket$ permutation module $R \llbracket G/H;\tau \rrbracket$. It suffices to prove $0 \in \mathcal{O}$. Observe first that $\mathcal{O}$ is non-empty, because $G \in (\cll'\clh_R)_\alpha \mathfrak{X}$ for some $\alpha$, and then $\alpha \in \mathcal{O}$ by hypothesis. So we need to show that if $0 \neq \beta \in \mathcal{O}$, there is some $\gamma < \beta$ such that $\gamma \in \mathcal{O}$.

So suppose $H \in (\cll'\clh_R)_\beta \mathfrak{X}$ and $H_R^i(G,R \llbracket G/H;\tau \rrbracket) \neq 0$. If $\beta$ is a limit, $H$ is in $(\cll'\clh_R)_\gamma \mathfrak{X}$ for some $\gamma < \beta$, and we are done; so assume $\beta$ is a successor ordinal. Now pick a direct system $\{H_j\}$ of subgroups of $H$ whose union is dense in $H$, with $H_j \in \clh_R(\cll'\clh_R)_{\beta-1}\mathfrak{X}$ for each $j$. Then we have a Type L system $\{R \llbracket G/H_j;\tau_j \rrbracket\}$ whose direct limit is $R \llbracket G/H;\tau \rrbracket$ by Corollary \ref{prodirlim+}, so we have an epimorphism $$\varinjlim_{PMod(R),j} H_R^i(G,R \llbracket G/H_j;\tau_j \rrbracket) \rightarrow H_R^i(G,R \llbracket G/H;\tau \rrbracket)$$ by Proposition \ref{epitypeI}: thus there is some $j$ such that $H_R^i(G,R \llbracket G/H_j;\tau_j \rrbracket) \neq 0$ too.

Suppose $R \llbracket G/H_j;\tau_j \rrbracket$ has twist homomorphism $\delta: H_j \rightarrow \{\pm 1\}$, and write $R'$ for a copy of $R$ on which $H_j$ acts by $h \cdot r = \delta(h)r$. Recall that $H_j \in \clh_R(\cll'\clh_R)_{\beta-1}\mathfrak{X}$; take a finite length signed permutation resolution $$0 \rightarrow P_n \rightarrow P_{n-1} \rightarrow \cdots \rightarrow P_0 \rightarrow R \rightarrow 0$$ of $R$ as a trivial $R \llbracket H_j \rrbracket$-module with stabilisers in $(\cll'\clh_R)_{\beta-1} \mathfrak{X}$, and apply induction $\Ind^G_{H_j}(- \hat{\otimes}_R R')$, where $- \hat{\otimes}_R R'$ is given the diagonal $H_j$ action, to get a sequence
\begin{align*}
0 &\rightarrow \Ind^G_{H_j} (P_n \hat{\otimes}_R R') \rightarrow \Ind^G_{H_j} (P_{n-1} \hat{\otimes}_R R') \rightarrow \cdots \\
&\rightarrow \Ind^G_{H_j} (P_0 \hat{\otimes}_R R') \rightarrow \Ind^G_{H_j} (R \hat{\otimes}_R R') \rightarrow 0
\end{align*}
which is exact by \cite[Theorem 6.10.8(c)]{R-Z}. Now $$\Ind^G_{H_j} (R \hat{\otimes}_R R') = R \llbracket G/H_j;\tau_j \rrbracket$$ by Lemma \ref{twist}; hence, by Lemma \ref{lesneq0}, $$H_R^{i+r}(G,\Ind^G_{H_j} (P_r \hat{\otimes}_R R')) \neq 0$$ for some $0 \leq r \leq n$. Now $\Ind^G_{H_j} (P_r \hat{\otimes}_R R')$ is a signed permutation module $P$ by Lemma \ref{indsperm}, so it has the structure of a Type H system, by Lemma \ref{typeIIsperm}, with components of the form $R \llbracket G/K;\tau' \rrbracket$, some $K \in (\cll'\clh_R)_{\beta-1} \mathfrak{X}$. By Proposition \ref{epitypeII}, $$\bigoplus_{PMod(R),K} H_R^{i+r}(G,R \llbracket G/K;\tau' \rrbracket) \rightarrow H_R^{i+r}(G,P)$$ is an epimorphism, so there is some $K$ such that $H_R^{i+r}(G,R \llbracket G/K;\tau' \rrbracket) \neq 0$. Since $K \in (\cllh_R)_{\beta-1} \mathfrak{X}$, this completes the inductive step of the proof.
\end{proof}

In particular, if the only $\mathfrak{X}$-subgroup of $G$ is the trivial one, by Lemma \ref{freeperm} there is some $n$ such that $H_R^n(G,R \llbracket G \rrbracket) \neq 0$. If $\mathfrak{X} = \mathfrak{F}$, we can say slightly more.

\begin{cor}
\label{neq0hf}
Suppose $G \in \widehat{\cll'\clh_R}\mathfrak{F}$ is of type $\FP_\infty$. Then there is some $n$ such that $H_R^n(G,R \llbracket G \rrbracket) \neq 0$.
\end{cor}
\begin{proof}
From Theorem \ref{neq0h} we know there is a finite $H \leq G$, a $\sigma$ and an $n$ such that $H_R^n(G,R \llbracket G/H;\sigma \rrbracket) \neq 0$. Pick an open normal subgroup $U \leq G$ such that $U \cap H = 1$: such a subgroup exists because the open normal subgroups of $G$ form a fundamental system of neighbourhoods of the identity by \cite[Lemma 2.1.1]{R-Z}. Then the Lyndon-Hochschild-Serre spectral sequence \cite[Theorem 4.2.6]{S-W} gives that $H_R^i(U,R \llbracket G/H;\sigma \rrbracket) \neq 0$, for some $i \leq n$. As a $U$-space, the stabiliser of the coset $gH, g \in G$, has the form $$U \cap H^{g^{-1}} = U^{g^{-1}} \cap H^{g^{-1}} = (U \cap H)^{g^{-1}} = 1;$$ hence $G/H$ is free as a $U$-space, so as an $R \llbracket U \rrbracket$-module $R \llbracket G/H;\sigma \rrbracket$ is finitely generated and free by Lemma \ref{freeperm}, and additivity gives $H_R^i(U,R \llbracket U \rrbracket) \neq 0$. Now
\begin{align*}
H_R^i(G,R \llbracket G \rrbracket) = &= H_R^i(G,Ind_U^G R \llbracket U \rrbracket) \\
&= H_R^i(G,Coind_U^G R \llbracket U \rrbracket) &&\text{ by \cite[(3.3.7)]{S-W}} \\
&= H_R^i(U,R \llbracket U \rrbracket)  \neq 0 &&\text{ by \cite[Theorem 10.6.5]{R-Z}},
\end{align*}
as required.
\end{proof}

As in \cite[Theorem A]{Krop}, there is no particular reason to restrict from $\Ext$-functors to group cohomology: all we need to know is that the first variable of these functors is of type $FP_\infty$ over $R \llbracket G \rrbracket$, and that it is projective on restriction to $R$. We sketch the proof of the theorem which follows from this observation; it is almost exactly the same as the proof of Theorem \ref{neq0h}.

\begin{thm}
\label{neq0ext}
Suppose $G \in \widehat{\cll'\clh_R} \mathfrak{X}$. Suppose $M \in PMod(R \llbracket G \rrbracket)$ is projective as an $R$-module, by restriction, and is of type $\FP_\infty$ over $R \llbracket G \rrbracket$. Then there is some subgroup $H \leq G$ which is in $\mathfrak{X}$, some signed $R \llbracket G \rrbracket$ permutation module $R \llbracket G/H;\sigma \rrbracket$ and some $n$ such that $\Ext_{R \llbracket G \rrbracket}^n(M,M \hat{\otimes}_R R \llbracket G/H;\sigma \rrbracket) \neq 0$.
\end{thm}
\begin{proof}
Replace $H_R^i(G,-)$ with $\Ext_{R \llbracket G \rrbracket}^i(M,-)$. Replace the signed permutation module coefficients $R \llbracket X \rrbracket$ of these functors with $M \hat{\otimes}_R R \llbracket X \rrbracket$, with the diagonal $G$-action. Then the proof goes through as before, after noting three things: that $M \hat{\otimes}_R -$ preserves Type L structures by Lemma \ref{typeI+}, that it preserves Type H structures by Lemma \ref{typeII+}, and that it preserves exactness of finite length signed permutation resolutions because finite length signed permutation resolutions of $R$ are $R$-split.
\end{proof}

Once again, if the only $\mathfrak{X}$-subgroup of $G$ is the trivial one, by Lemma \ref{freeperm} there is some $n$ such that $\Ext_{R \llbracket G \rrbracket}^n(M,M \hat{\otimes}_R R \llbracket G \rrbracket) \neq 0$. There is also a result corresponding to Corollary \ref{neq0hf}.

\section{Totally Disconnected Polish \texorpdfstring{$R$}{R}-modules}
\label{tdmod}

To use the results of Section \ref{main} to derive information about the group structure of $G$, as in \cite[Section 2]{Krop3}, we need the work of \cite{Boggi}, which provides the Lyndon-Hochschild-Serre spectral sequence we need.

All profinite groups and rings in this and the next section will be assumed to be countably based, unless stated otherwise -- for background on countably based profinite groups, see \cite[Section 2.6]{R-Z}. Note in particular that this class includes all finitely generated groups, and hence all pro-$p$ groups of type $\FP_1$ over $\mathbb{Z}_p$ (and all subgroups of such groups), by \cite[Remark 3.5(c)]{Myself2}. By \cite[Remark 3.5(a)]{Myself2}, in fact all prosoluble groups of type $\FP_1$ over $\hat{\mathbb{Z}}$ are finitely generated. On the other hand, the following example shows that a group in $\widehat{\cll'\clh_{\hat{\mathbb{Z}}}}\mathfrak{F}$ need not be countably based even if it is of type $\FP_1$. This example is adapted from \cite[Example 2.6]{Damian}; the approach is the same, but we construct groups which are not countably based.

\begin{example}
Consider a product of copies of $A_5$, the alternating group on $5$ letters, indexed by a set $I$. Suppose $I$ has cardinality $\aleph_\alpha$ for some ordinal $\alpha$. Since $A_5$ is simple, the finite quotients of $\prod_I A_5$ are all $\prod_{i=1}^n A_5$. By \cite[Example 2.6]{Damian}, the minimal number of generators of $\prod_{i=1}^n A_5$ tends to $\infty$ as $n$ does, but the augmentation ideal $\ker(\hat{\mathbb{Z}} \llbracket \prod_{i=1}^n A_5 \rrbracket \rightarrow \hat{\mathbb{Z}})$ is $2$-generated for all $n$. It follows by \cite[Theorem 2.3]{Damian} that $\prod_I A_5$ is of type $\FP_1$ over $\hat{\mathbb{Z}}$.

Since $A_5$ is discrete, the family $F$ of neighbourhoods of $1$ in $\prod_I A_5$ of the form $$(\prod_{\{i \in I: i \neq i_1, \ldots, i_t\}} A_5) \times \{1\}_{i_1} \times \cdots \times \{1\}_{i_t},$$ for any $i_1, \ldots, i_t \in I$, is a fundamental system of neighbourhoods of $1$ in $\prod_I A_5$. Since $I$ has cardinality $\aleph_\alpha$, $F$ does too. Hence by \cite[Proposition 2.6.1]{R-Z} $\prod_I A_5$ has weight $\aleph_\alpha$. In particular, for $\alpha > 0$, $\prod_I A_5$ is not countably based.

Finally, to see that $\prod_I A_5 \in \widehat{\cll'\clh_{\hat{\mathbb{Z}}}}\mathfrak{F}$, the easiest way is to note that $\bigoplus_I A_5$ is dense in $\prod_I A_5$, and $\bigoplus_I A_5$ is clearly locally finite, so we have $\prod_I A_5 \in \cll'\mathfrak{F}$.
\end{example}

\begin{question}
Are there profinite groups of type $\FP_2$ which are not finitely generated?
\end{question}

We work now with the categories of modules developed in \cite{Boggi}: given $R \in PRng$, we call a topological $R$-module $M$ a \emph{totally disconnected} (or \emph{t.d.}) \emph{Polish $R$-module} if $M$ is complete, Hausdorff, has a basis of neighbourhoods of $0$ consisting of open submodules and is second countable. Thus all countably based profinite $R$-modules are t.d. Polish $R$-modules. We write $Mod^{t.d.}(R)$ for the category of t.d. Polish $R$-modules and continuous module homomorphisms.

It is possible to define, for a countably based commutative profinite ring $R$ and a countably based profinite group $G$, the homology and cohomology groups of $G$ over $R$ with coefficients in $Mod^{t.d.}(R \llbracket G \rrbracket)$. In particular, we need the following result.

\begin{lem}
\label{td=pro}
Write $H_R^n(G,-)$ for the usual $n$th cohomology group of $G$ with profinite coefficients, and $\underline{H}_R^n(G,M)$ for the $n$th cohomology group of $G$ with coefficients in $Mod^{t.d.}(R \llbracket G \rrbracket)$. Suppose $M \in Mod^{t.d.}(R \llbracket G \rrbracket)$ is profinite. Then, for all $n \geq 0$, $\underline{H}_R^n(G,M) = H_R^n(G,M)$.
\end{lem}
\begin{proof}
Both can be calculated using the bar resolution (see \cite[p.28]{Boggi}).
\end{proof}

As a result of this, we will write $H_R^n(G,M)$ for both cohomology theories.

\section{Soluble Groups}
\label{solgp}

We now establish some properties of nilpotent profinite groups; here we take nilpotent to mean that a group's (abstract) upper central series becomes the whole group after finitely many steps. All these results correspond closely to known ones in the abstract case, but there doesn't seem to be a good profinite reference, so they are included here.

\begin{lem}
Each term in the upper central series of a profinite group is closed.
\end{lem}
\begin{proof}
We show first that $Z_1(G)$ is closed. For each $g \in G$, the centraliser $C_G(g)$ of $g$ is the inverse image of $1$ in the continuous map $G \rightarrow G, x \mapsto [g,x]$, so it is closed. Then $Z_1(G) = \bigcap_{g \in G} C_G(g)$ is closed.

Now we use induction: suppose $Z_{i-1}(G)$ is closed. We know the centre of $G/Z_{i-1}(G)$ is closed, and $Z_i(G)$ is the preimage of $Z(G/Z_{i-1}(G))$ under the projection $G \rightarrow G/Z_{i-1}(G)$; hence $Z_i(G)$ is closed too.
\end{proof}

Thus nilpotent profinite groups are exactly the profinite groups which are nilpotent as abstract groups.

Since $G$ is nilpotent as an abstract group, write $$G = C_{abs}^1(G) \rhd C_{abs}^2(G) \rhd \cdots$$ for the terms of the abstract lower central series of $G$, and define the profinite upper central series by $C^n(G) = \overline{C_{abs}^n(G)}$. Each $C^n(G)$ is normal, as the closure of a normal subgroup. Moreover, since $[G,C_{abs}^n(G)] = C_{abs}^{n+1}(G)$, we have $\overline{[G,C^n(G)]} = \overline{[G,C_{abs}^n(G)]} = C^{n+1}(G)$. If $G$ has nilpotency class $k$, $C_{abs}^{k+1}(G)=1 \Rightarrow C^{k+1}(G)=1$. In particular $C^n(G)$ has nilpotency class $k+1-n$, for $n \leq k$.

\begin{lem}
\label{fgnilpot}
Suppose $G$ is a finitely generated nilpotent profinite group of nilpotency class $k$. Then every subgroup $H \leq G$ is finitely generated.
\end{lem}
\begin{proof}
Let $X$ be a finite generating set for $G$. Write $G^{abs}$ for the dense subgroup of $G$ generated abstractly by $X$, and $C_{abs}^n(G^{abs})$ for the terms in its (abstract) upper central series. Now $$\overline{C_{abs}^2(G^{abs})} = \overline{[G^{abs},G^{abs}]} = \overline{[\overline{G^{abs}},\overline{G^{abs}}]} = \overline{[G,G]} = C^2(G).$$ By \cite[5.2.17]{Robinson}, $C_{abs}^2(G^{abs})$ is abstractly finitely generated, so its closure $C^2(G)$ is topologically finitely generated.

We now prove the lemma by induction on $k$: when $k=1$, $G$ is abelian, and we are done by \cite[Proposition 8.2.1]{Wilson}. So suppose the result holds for $k-1$. Since $C^2(G)$ has class $k-1$ and $G/C^2(G)$ has class $1$, by hypothesis $H \cap C^2(G)$ and $H/(H \cap C^2(G))$ are both finitely generated, and hence $H$ is too.
\end{proof}

\begin{lem}
\label{tfnilpot}
Suppose $G$ is a finitely generated torsion-free nilpotent profinite group. Then $G$ is poly-(torsion-free procyclic).
\end{lem}
\begin{proof}
Suppose $G$ has nilpotency class $k$. Consider the upper central series of $G$, $$1 \lhd Z_1(G) \lhd \cdots \lhd Z_k(G)=G.$$ If we show that every factor $Z_{j+1}(G)/Z_j(G)$ is torsion-free, then it will follow by Lemma \ref{fgnilpot} that every $Z_{i+1}(G)/Z_i(G)$ is finitely generated torsion-free abelian, hence poly-(torsion-free procyclic), and we will be done. Both these facts are known in the abstract case.

Clearly $Z_1(G)$ is torsion-free, and we use induction on $k$, on the hypothesis that $Z_{j+1}(G)/Z_j(G)$ is torsion-free whenever $G$ of nilpotency class $k$ has $Z_1(G)$ torsion-free. $k=1$ is trivial. For $k>1$, we show first that $Z_2(G)/Z_1(G)$ is torsion-free by showing, for each $1 \neq x \in Z_2(G)/Z_1(G)$, that there is some $\phi \in Hom(Z_2(G)/Z_1(G), Z_1(G))$ such that $\phi(x) \neq 1$. Then the result follows because $Z_1(G)$ is torsion-free. So pick a preimage $x'$ of $x$ in $Z_2(G)$. $x' \notin Z_1(G)$, so there is some $g \in G$ such that $1 \neq [g,x'] \in Z_1(G)$. Now define $$\phi': Z_2(G) \rightarrow Z_1(G), y \mapsto [g,y];$$ note that $\phi'$ is a homomorphism, because $$[g,y_1y_2]=[g,y_1][y_1,[g,y_2]][g,y_2]=[g,y_1][g,y_2],$$ since $[g,y_2] \in Z_1(G) \Rightarrow [y_1,[g,y_2]]=1$. $\phi'(x') \neq 1$, and $Z_1(G) \leq ker(\phi')$, so this induces $\phi: Z_2(G)/Z_1(G) \rightarrow Z_1(G)$ such that $\phi(x) \neq 1$, as required.

By hypothesis, the centre of $G/Z_1(G)$ being torsion-free implies that $$Z_{j+1}(G)/Z_j(G) = Z_j(G/Z_1(G))/Z_{j-1}(G/Z_1(G))$$ is torsion-free, for each $j$.
\end{proof}

Recall that a profinite group $G$ is said to have \emph{finite rank} if there is some $r$ such that every subgroup $H$ of $G$ is generated by $r$ elements. By \cite[Proposition 8.1.1]{Wilson}, the class of profinite groups of finite rank is closed under taking subgroups, quotients and extensions.

\begin{lem}
\label{poincare}
Let $G$ be a finitely generated torsion-free nilpotent pro-$p$ group of nilpotency class $k$, and let $F \in Mod^{t.d.}(\mathbb{Z}_p \llbracket G \rrbracket)$ be a free profinite $\mathbb{Z}_p \llbracket G \rrbracket$-module. If $H_{\mathbb{Z}_p}^n(G,F) \neq 0$, then $k \leq n$ and $G$ has rank $\leq n$.
\end{lem}
\begin{proof}
Suppose $G$ has Hirsch length $m$. Note that $k \leq m$ by Lemma \ref{tfnilpot}. Then Lemma \ref{tfnilpot} gives also that $cd_{\mathbb{Z}_p} G = m$, by \cite[Proposition 4.3.1]{S-W}, and that $G$ is a Poincar\'{e} duality group in dimension $m$ by \cite[Theorem 5.1.9]{S-W}. Hence $H_{\mathbb{Z}_p}^i(G,F) = 0$ for $i \neq m$, and so $m=n \geq k$. $G$ is built, by extensions, out of $n$ groups of rank $1$, so $G$ has rank $\leq n$, by repeated applications of \cite[Proposition 8.1.1(b)]{Wilson}.
\end{proof}

\begin{lem}
\label{subnormalcohom}
Let $G$ be a profinite group, and suppose $M \in Mod^{t.d.}(R \llbracket G \rrbracket)$ such that
$H_R^n(G, M) \neq 0$ for some $i$. If $H$ is a subnormal subgroup of $G$, there is some $i \leq n$ such
that $H_R^i(H, M) \neq 0$.
\end{lem}
\begin{proof}
For $H$ normal, we use the Lyndon-Hochschild-Serre spectral sequence \cite[6.17]{Boggi}. For $H$ subnormal, we have a sequence $$H = G_m \lhd G_{m-1} \lhd \cdots \lhd G_0 = G,$$ and we use the spectral sequence repeatedly to show that for each $0 \leq k \leq m$ there is some $n_k \leq n$ such that $H_R^{n_k}(G_k, M) \neq 0$.
\end{proof}

\begin{lem}
\label{subnormal}
Every subgroup $H$ of a profinite nilpotent group $G$ is subnormal.
\end{lem}
\begin{proof}
Consider the upper central series of $G$, $$1 \lhd Z_1(G) \lhd \cdots \lhd Z_k(G)=G.$$ Then $$H \leq HZ_1(G) \leq \cdots \leq HZ_k(G)=G$$ gives a subnormal series for $H$: to see that $HZ_i(G)$ is normal in $HZ_{i+1}(G)$, note that $H$ clearly normalises $HZ_i(G)$, and $Z_{i+1}(G)$ does because $$[Z_{i+1}(G),HZ_i(G)] \leq [Z_{i+1}(G),G] \leq Z_i(G) \leq HZ_i(G),$$ so $HZ_{i+1}(G)$ does too.
\end{proof}

For abstract groups, the Fitting subgroup is defined to be the join of the nilpotent normal subgroups. \cite[Section 8.4]{Wilson} defines a profinite Fitting subgroup of a profinite group $G$ as the inverse limit of the Fitting subgroups of the finite quotients of $G$; this is not the definition we will use. Instead we define the \emph{abstract Fitting subgroup} to be the abstract subgroup generated by the nilpotent normal closed subgroups of $G$.

\begin{prop}
\label{fitting}
Let $G$ be a torsion-free pro-$p$ group, $N$ its abstract Fitting subgroup, and $\bar{N} \geq N$ the closure of $N$ in $G$. If there is some free profinite $\mathbb{Z}_p \llbracket G \rrbracket$-module $F \in Mod^{t.d.}(\mathbb{Z}_p \llbracket G \rrbracket)$ such that $H_{\mathbb{Z}_p}^n(G,F) \neq 0$, then $\bar{N}$ is nilpotent of nilpotency class and rank $\leq n$.
\end{prop}
\begin{proof}
We claim the join of any finite collection $N_1, \ldots, N_m$ of nilpotent normal closed subgroups of $G$ is nilpotent, normal and closed. Consider the abstract join $N'$ of these as subgroups of an abstract group: then it is known that $N'$ is nilpotent and normal. Moreover, because all the subgroups are normal, $N' = N_1 \cdots N_m$, which is closed in $G$, so $N'$ is the join of $N_1, \ldots, N_m$ as profinite subgroups of $G$, and we are done.

So we can see $N$ as the directed union of the nilpotent normal subgroups of $G$. Suppose $H$ is a finitely generated subgroup of $G$, generated by finitely many elements of $N$. Then $H$ is contained in a nilpotent normal subgroup of $G$ (and so it is also contained in $N$); hence it is finitely generated torsion-free nilpotent pro-$p$, and subnormal by Lemma \ref{subnormal}. So by Lemma \ref{subnormalcohom} $H_{\mathbb{Z}_p}^i(H,F) \neq 0$ for some $i \leq n$, and hence $H$ has nilpotency class and rank $\leq n$ by Lemma \ref{poincare}.

This holds for every finitely generated subgroup of $N$, so $N$ is nilpotent of class $\leq n$. Thus the continuous map $$\overbrace{N \times N \times \cdots \times N}^{n+1} \rightarrow \overbrace{[N,[N,[\cdots,N]\cdots]]}^{n+1}$$ has image $1$, and by continuity its closure $$\overbrace{\bar{N} \times \bar{N} \times \cdots \times \bar{N}}^{n+1} \rightarrow \overbrace{[\bar{N},[\bar{N},[\cdots,\bar{N}]\cdots]]}^{n+1}$$ also has image $1$. Therefore $\bar{N}$ is nilpotent of class $\leq n$ too, and it is normal because $N$ is, so by definition of $N$ we have $\bar{N} \leq N \Rightarrow \bar{N}=N$. Finally, we have shown that every finitely generated subgroup of $\bar{N}$ has rank $\leq n$, so $\bar{N}$ does too.
\end{proof}

One of the useful properties of the Fitting subgroup for abstract soluble groups is that it contains its own centraliser. The easiest way to show that the same property holds for profinite soluble groups is to show that the two are the same.

\begin{lem}
\label{fittingcentraliser}
Let $G$ be a profinite group and $N$ its abstract Fitting subgroup. Write $G^{\text{abs}}$ for $G$ considered as an abstract group, and let $N^{\text{abs}}$ be the Fitting subgroup of $G^{\text{abs}}$. Then, as (abstract) subgroups of $G$, $N = N^{\text{abs}}$. Thus, for $G$ soluble, $N$ contains its own centraliser in $G$.
\end{lem}
\begin{proof}
Every nilpotent normal closed subgroup $H$ of $G$ is a nilpotent normal abstract subgroup, so every such $H$ is contained in $N^{\text{abs}}$, and hence so is $N$, i.e. $N \leq N^{\text{abs}}$.

Suppose instead that $H$ is a nilpotent normal abstract subgroup of nilpotency class $i$. Then the closure $\bar{H}$ is a normal closed subgroup of $G$. As before, the continuous map $$\overbrace{H \times H \times \cdots \times H}^{i+1} \rightarrow \overbrace{[H,[H,[\cdots,H]\cdots]]}^{i+1}$$ has image $1$, and by continuity its closure $$\overbrace{\bar{H} \times \bar{H} \times \cdots \times \bar{H}}^{i+1} \rightarrow \overbrace{[\bar{H},[\bar{H},[\cdots,\bar{H}]\cdots]]}^{i+1}$$ also has image $1$, so $\bar{H}$ is nilpotent. Hence $H \leq \bar{H} \leq N$, and therefore $N^{\text{abs}} \leq N$.
\end{proof}

The following result corresponds roughly to \cite[Theorem B]{Krop3}, and answers \cite[Open Question 6.12.1]{R-Z} in the torsion-free case.

\begin{thm}
\label{frank}
Let $G$ be a virtually torsion-free soluble pro-$p$ group of type $FP_\infty$ over $\mathbb{Z}_p$. Then $G$ has finite rank.
\end{thm}
\begin{rem}
We do not need to assume in addition that $G$ is countably based; it follows from the $FP_\infty$ hypothesis by \cite[Remark 3.5(c)]{Myself2}.
\end{rem}
\begin{proof}
We can assume $G$ is torsion-free: if it isn't, take a finite index torsion-free subgroup. Write $N$ for the abstract Fitting subgroup of $G$. Then by Corollary \ref{neq0hf} there is some $n$ such that $H_{\mathbb{Z}_p}^n(G,\mathbb{Z}_p \llbracket G \rrbracket) \neq 0$. Now $\mathbb{Z}_p$ and $G$ are countably based, so $\mathbb{Z}_p \llbracket G \rrbracket$ is too, and hence it is a t.d. Polish $\mathbb{Z}_p \llbracket G \rrbracket$-module; thus Proposition \ref{fitting} gives us that $N$ is closed and has finite rank. Then, writing $C_G(N)$ for the centraliser of $N$ in $G$, we have a monomorphism $G/C_G(N) \rightarrow Aut(N)$, and $Aut(N)$ has finite rank by \cite[Theorem 5.7]{DDMS}, so $G/C_G(N)$ has finite rank too. But by Lemma \ref{fittingcentraliser} $C_G(N) \leq N$ has finite rank, so $G$ does.
\end{proof}

We observe that, by \cite[Proposition 4.2]{Myself2}, Theorem \ref{frank} has the following converse: Suppose $G$ is a soluble pro-$p$ group of finite rank. Then $G$ is virtually torsion-free of type $\FP_\infty$ over $\mathbb{Z}_p$.


\end{document}